\documentclass[11pt,a4paper,openany, reqno]{amsart}
\usepackage{mathrsfs}

\usepackage{amsmath,amsfonts,verbatim}
\usepackage{latexsym}
\usepackage{amssymb,leftidx}
\usepackage{extarrows}
\usepackage{bbm}
\usepackage{overpic}
\usepackage{color}
\usepackage{epsfig}
\usepackage{subfigure}
\usepackage{tikz}

\usepackage[backref=page]{hyperref}

\usepackage{amssymb}
\usepackage{mathrsfs}
\usepackage{amscd}
\usepackage{cite}

\newcommand{\qtq}[1]{\quad\text{#1}\quad}

\newcommand\R{\mathbb{R}}
\newcommand\C{\mathbb{C}}
\newcommand\Z{\mathbb{Z}}

\newcommand{\pa}{\partial}

\newcommand\LL{\mathcal{L}}
\newcommand\A{\bf A}

\usepackage{graphicx}
\usepackage{float}

\numberwithin{equation}{section}
\newtheorem{proposition}{Proposition}[section]
\newtheorem{definition}{Definition}[section]
\newtheorem{lemma}{Lemma}[section]
\newtheorem{theorem}{Theorem}[section]

\newtheorem{remark}{Remark}[section]

\begin{document}
\title[Dispersive estimates for critical waves estimates]{Dispersive estimates for 2D-wave equations with critical potentials}

\author{Luca Fanelli}
\address{Ikerbasque and Universidad del Pa\'is Vasco, Departamento de Matem\'aticas, Bilbao, Spain}
\email{fanelli@mat.uniroma1.it}

\author{Junyong Zhang}
\address{Department of Mathematics, Beijing Institute of Technology, Beijing 100081; Department of Mathematics, Cardiff University, UK}
\email{zhang\_junyong@bit.edu.cn; ZhangJ107@cardiff.ac.uk}

\author{Jiqiang Zheng}
\address{Institute of Applied Physics and Computational Mathematics, Beijing 100088}
\email{zhengjiqiang@gmail.com; zheng\_jiqiang@iapcm.ac.cn}

\maketitle


\begin{abstract}
We study the 2D-wave equation with a scaling-critical electromagnetic potential. This problem is doubly critical, because of the scaling invariance of the model and the singularities of the potentials, which are not locally integrable. In particular, the diamagnetic phenomenon allows to consider negative electric potential which can be singular in the same fashion as the inverse-square potential. We prove sharp time-decay estimates in the purely magnetic case, and Strichartz estimates for the complete model, involving a critical electromagnetic field.
\end{abstract}

\begin{center}
 \begin{minipage}{120mm}
   { \small {\bf Key Words:  Decay estimates, Strichartz estimates, Aharonov-Bohm magnetic field, scaling-critical electromagnetic potential, wave equation}
      {}
   }\\
    { \small {\bf AMS Classification:}
      { 42B37, 35Q40.}
      }
 \end{minipage}
 \end{center}

\maketitle 

\section{Introduction}\label{sec:intro}

Let us consider the following initial-value problem for the wave equation
 \begin{align} \label{wave}
\begin{cases}    \partial_{tt}u+\LL_{{\A},a} u= 0,\quad
(t,x)\in\mathbb{R}\times\mathbb{R}^2,
\\
u(0,x)=f(x),\quad \partial_t u(0,x)=g(x).
\end{cases}
\end{align}
Here, the operator $\LL_{{\A},a}$ is defined by
\begin{equation}\label{LAa}
\mathcal{L}_{{\A},a}=\Big(i\nabla+\frac{{\A}(\hat{x})}{|x|}\Big)^2+\frac{a(\hat{x})}{|x|^2},
\end{equation}
where $\hat{x}=\tfrac{x}{|x|}\in\mathbb{S}^1$, $a\in L^{\infty}(\mathbb{S}^{1},\mathbb{R})$ and ${\A}\in W^{1,\infty}(\mathbb{S}^1;\R^2)$  satisfies the transversality condition
\begin{equation}\label{eq:transversal}
{\A}(\hat{x})\cdot\hat{x}=0,
\qquad
\text{for all }x\in\R^2.
\end{equation}
Our two main examples are the following:
\begin{itemize}
\item
the {\it Aharonov-Bohm} potential
\begin{equation}\label{ab-potential}
a\equiv0,
\qquad
{\A}(\hat{x})=\alpha\Big(-\frac{x_2}{|x|},\frac{x_1}{|x|}\Big),\quad \alpha\in\R,
\end{equation}
introduced in \cite{AB}, in the context of Schr\"odinger dynamics, to show that scattering effects can even occur in regions in which the electromagnetic field is absent (see also \cite{PT89});
\item
the {\it inverse-square} potential
\begin{equation}\label{eq:inversesquare}
{\A}\equiv0,
\qquad
a(\hat{x})\equiv a>0.
\end{equation}
\end{itemize}
Throughout this paper, we will always assume that
\begin{equation}\label{equ:condassa}
  \|a_-\|_{L^\infty(\mathbb{S}^1)}<\min_{k\in\Z}\{|k-\Phi_{\A}|\}^2,
  \qquad
  \Phi_{\A}\notin\Z,
\end{equation}
where $a_-:=\max\{0,-a\}$ is the negative part of $a$, and $\Phi_{\A}$ is the total flux along the sphere
\begin{equation}\label{equ:defphia1}
  \Phi_{\A}=\frac{1}{2\pi}\int_0^{2\pi}A(\theta)\;d\theta,
\end{equation}
where $A(\theta)$ is defined by \eqref{equ:Athetadef} below. Indeed, thanks to the Hardy inequality
\begin{equation}\label{equ:ghardyLW}
  \min_{k\in\Z}\{|k-\Phi_{\A}|\}^2\int_{\R^2}\frac{|f|^2}{|x|^2}\;dx\leq \int_{\R^2}|\nabla_{\A}f|^2\;dx,
  \qquad
  \nabla_{{\A}}:=i\nabla+{\A},
\end{equation}
(see \cite{LW}, and \cite[cf. (27)]{FFT}), thanks to assumption \eqref{equ:condassa} the Hamiltonian $\mathcal{L}_{{\A},a}$ can be defined as a self-adjoint operator on $L^2$, via Friedrichs' Extension Theorem (see e.g. \cite[Thm. VI.2.1]{K} and \cite[X.3]{RS}), on the natural form domain, which in 2D turns out to be equivalent to
$$
\mathcal D(\mathcal L_{\A,a})\simeq H^1_{{\A},a}:=\left\{f\in L^2(\R^2;\C):\int_{\R^2}\left(\left|\nabla_{{\A}} f\right|^2+\frac{|f(x)|^2}{|x|^2}|a(\hat x)|\right)\,dx<+\infty\right\}.
$$
We stress that in 2D, we have $H_{{\A},a}^1(\R^2)\subset H^1(\R^2)$, and the inclusion is strict, because of the non integrable singularities of the potentials (see \cite[cf. Lemma 23 - (ii)]{FFT} for details).
Therefore, the Spectral Theorem allows us to consider the dispersive propagators $e^{it\sqrt{\mathcal L_{{\A},a}}}, \cos(t\sqrt{\mathcal L_{{\A},a}}), \frac{\sin(t\sqrt{\mathcal L_{{\A},a}})}{\sqrt{\mathcal L_{{\A},a}}}$, as one-parameter groups of operators on $L^2$. In particular, the unique solution to \eqref{wave} can be represented by
\begin{equation}\label{eq:solution}
u(t,\cdot)=\cos(t\sqrt{\mathcal L_{{\A},a}})f(\cdot)+\frac{\sin(t\sqrt{\mathcal L_{{\A},a}})}{(\sqrt{\mathcal L_{{\A},a}})}g(\cdot).
\end{equation}
One of the main mathematical features of the wave equation \eqref{wave} is the scaling invariance, namely
$$
u_\lambda(t,x):=\lambda^{2}u\big(\tfrac t\lambda,\tfrac x\lambda\big)
\qquad
\Rightarrow
\qquad
(\partial_{tt}+\mathcal L_{{\A,}a})u_\lambda(t,x)= (\partial_{tt}u+\mathcal L_{{\A,}a}u)\big(\tfrac t\lambda,\tfrac x\lambda\big).
$$
This makes the dispersive evolution in \eqref{wave} critical with respect to a large class of dispersive estimates, as e.g. time-decay, Strichartz and local smoothing. The validity of such properties has been object of deep investigation in the last decades, due to their relevance in the description of linear and nonlinear dynamics. We now briefly sketch the state of the art about these problems.\vspace{0.2cm}

\noindent
{\bf Purely electric case} ${\A}\equiv0$. The first available results are due to Burq, Planchon, Stalker, and Tahvildar-Zadeh in \cite{BPSS, BPST}, in which they proved the validity of Strichartz estimates for the Schr\"odinger and wave equations, in space dimension $n\geq2$.
Assumption \eqref{equ:condassa} is replaced by the natural one which involves the usual Hardy inequality. For the inverse-square potential $a(\hat{x})\equiv a\in\R$ it reads, in dimension $n\geq3$, as $a>-(n-2)^2/4$, while $a\geq0$ is needed in dimension $n=2$, due to the lack of Hardy inequality. More recently, Mizutani treated in \cite{M} the analog problem for Schr\"odinger for the critical inverse-square $a=-(n-2)^2/4$, in dimension $n\geq3$. Later, Fanelli, Felli, Fontelos, and Primo proved in \cite{FFFP, FFFP1} investigated the validity of the time-decay estimate for the Schr\"odinger evolution, and proved that it holds, in some specific cases, including the inverse square potential. A quite interesting remark in \cite{FFFP, FFFP1} is that the usual time-decay for the Schr\"odinger equation does not hold in the range $-\tfrac{(n-2)^2}4<a<0$, while Strichartz estimates are true.

\noindent
{\bf Electromagnetic case.} If a magnetic field is present, the picture is more unclear. After a sequel of papers (see \cite{CS,DF, DFVV, EGS1, EGS2, S} and the references therein) in which time-decay or Strichartz estimates are studied, with subcritical magnetic potentials, in \cite{FFFP1}, the author noticed that the space dimension $n=2$ is very peculiar for this kind of problems. Indeed, from one side the critical potential ${\A}/|x|$ is not in $L^2_{loc}$, which means that the domain $H^1_{{\A},0}$ is strictly contained in $H^1$; from the other side, since the associated spherical problem is 1-dimensional, several explicit expansions are available, leading to quite complete results. For examples,  the time-decay estimate
\begin{equation}\label{eq:decayshro}
\|e^{it\LL_{{\A},a}}\|_{L^1(\R^2)\to L^\infty(\R^2)}\lesssim |t|^{-1}
\end{equation}
for the Schr\"odinger equation is proved in \cite{FFFP}, provided \eqref{eq:transversal} holds, and $a(\hat{x})>0$. This implies Strichartz estimates for $e^{it\LL_{{\A},a}}$, by the usual Keel-Tao argument \cite{KT}. We also mention that the behavior ${\A}\sim \/|x|$ is known to be critical for the validity of Strichartz estimates, as proved e.g. in \cite{FG} in the case of the Schr\"odinger equation.
\vspace{0.2cm}

\noindent
In view of the above comments, the aim of this paper is to prove time decay and Strichartz estimates for equation \eqref{wave}, in the same fashion as in \cite{FFFP} for the Schr\"odinger equation. In order to do this, let us introduce some preliminary notations. In the following, the Sobolev spaces will be denoted by
\begin{align}\label{def:sobolev}
& \dot H^{s}_{{\A},a}(\R^2):=\mathcal L_{{\A},a}^{-\frac s2}L^2(\R^2),
\qquad
\dot H^s(\R^2):=\dot H^s_{0,0}(\R^2),
\\
&
H^s_{{\A},a}(\R^2) :=L^2(\R^2)\cap\dot H^{s}_{{\A},a}(\R^2),
\qquad
H^s(\R^2):= H^s_{0,0}(\R^2).
\nonumber
\end{align}
Analogously, we define the distorted Besov spaces as follows.  Let $\varphi\in C_c^\infty(\mathbb{R}\setminus\{0\})$, with $0\leq\varphi\leq 1$, $\text{supp}\,\varphi\subset[1/2,1]$, and
\begin{equation}\label{dp}
\sum_{j\in\Z}\varphi(2^{-j}\lambda)=1.
\end{equation}

\begin{definition}[Magnetic Besov spaces associated with $\mathcal{L}_{{\A},0}$] For $s\in\R$ and $1\leq p,r\leq\infty$, the norm of $\|\cdot\|_{\dot{\mathcal{B}}^s_{p,r,\A}(\R^2)}$ is given by
\begin{equation}\label{Besov}
\|f\|_{\dot{\mathcal{B}}^s_{p,r,\A}(\R^2)}=\Big(\sum_{j\in\Z}2^{jsr}\|\varphi_j(\sqrt{\LL_{{\A},0}})f\|_{L^p(\R^2)}^r\Big)^{1/r}.
\end{equation}
In particular, for $p=r=2$, we have
\begin{equation}\label{Sobolev}
\|f\|_{\dot{H}^s_{{\A},0}(\R^2)}:=\big\|\mathcal L_{{\A},0}^{\frac s2}f\big\|_{L^2(\R^2)}=\Big\|\Big(\sum_{j\in\Z}2^{2js}|\varphi_j(\sqrt{\mathcal{L}_{{\A},0}})f|^2\Big)^{1/2}
\Big\|_{L^2(\R^2)}=\|f\|_{\dot{\mathcal{B}}^s_{2,2,\A}(\R^2)}.
\end{equation}
\end{definition}
\noindent
Our first result is concerned with the time decay property for the propagator $\frac{\sin(t\sqrt{\mathcal L_{{\A},0}})}{\sqrt{\mathcal L_{{\A},0}}}$.
\begin{theorem}\label{thm:dispersive} Let ${\A}\in W^{1,\infty}(\mathbb{S}^{1},\mathbb{R}^2)$, and assume \eqref{eq:transversal}.
There exists a constant $C>0$ such that, for any
\begin{equation}\label{dis-w}
\Big\|\frac{\sin(t\sqrt{\mathcal L_{{\A},0}})}{\sqrt{\mathcal L_{{\A},0}}}f\Big\|_{L^\infty(\R^2)}\leq C |t|^{-1/2}\|f\|_{\dot{\mathcal{B}}^{1/2}_{1,1,{\A}}(\R^2)}.
\end{equation}
\end{theorem}
\begin{remark}
Theorem \ref{thm:dispersive} is analog to the main result in \cite{FFFP} for the Schr\"odinger equation. In that case, a crucial role is played by the pseudoconformal invariance of the Schr\"odinger equation, which together with a suitable transformation permits to pass to a Hamiltonian with an explicit, purely discrete spectrum, obtaining a nice representation formula for the solution in the physical space. A similar argument is not available for the wave equation, for which the natural conformal invariance leads to a much more complicate picture.
For this reason, we use a different strategy to prove Theorem \ref{thm:dispersive}, which is based on a representation of the fundamental solution, inspired to \cite{BFM, CT1}. The difficult question about the validity of \eqref{dis-w} for $e^{it\sqrt{\mathcal L_{{\A},a}}}$, with $a\neq0$ remains open, although we conjecture the answer is positive, provided  $a\geq0$.
\end{remark}

\noindent
As a consequence of the previous result, we can now study the validity of Strichartz estimates for the electromagnetic propagator $e^{it\sqrt{\mathcal L_{{\A},a}}}$, treating it as a perturbation of the purely magnetic operator $e^{it\sqrt{\mathcal L_{{\A},0}}}$.
We say $(q,r)$ is a (2D)-\emph{wave-admissible pair}, if
\begin{equation}\label{adm-p-w}
(q,r)\in\Lambda_s^W:=\big\{(q,r)\in [2,\infty]\times[2,\infty), \; \tfrac2q+\tfrac1r\leq \tfrac12, \;s=2\big(\tfrac12-\tfrac1r\big)-\tfrac1q\big\},\;s\in\R.
\end{equation}
We remark that $0\leq s<1$, otherwise the set $\Lambda_s^W$ is empty (see Figure \ref{figure1}). \vspace{0.2cm}

It is well known by \cite{KT} that there exists a constant $C>0$ such that the solution to the free wave equation $
u(t,\cdot):=\cos(t\sqrt{-\Delta})f(\cdot)+\frac{\sin(t\sqrt{-\Delta})}{\sqrt{-\Delta}}g(\cdot)$ satisfies the Strichartz estimate
$$
\|u\|_{L^q_t(\R;L^r(\R^2))}\leq C \big(\|f\|_{\dot{H}^s(\R^2)}+\|g\|_{\dot{H}^{s-1}(\R^2)}\big),
$$
for any wave-admissible pair $(q,r)$.

\begin{center}
 \begin{tikzpicture}[scale=1]\label{figure1} 

\draw[->] (-0.2,0) -- (3.5,0) node[anchor=north] {$\frac{1}{q}$};
\draw[->] (0,-0.2) -- (0,3.5)  node[anchor=east] {$\frac{1}{r}$};

\draw (-0.1,0) node[anchor=north] {O}
       (2.5,0) node[anchor=north] {$\frac12$}
       (1.25,0) node[anchor=north] {$\frac14$}
       (2,2.5)  node[anchor=north] {$\frac2q+\frac1r=\frac12$}
       (2,1.6)  node[anchor=north] {$\frac1q+\frac2r=\frac12$};
\draw  (0, 2.5) node[anchor=east] {$\frac12$}
       (0,1.25) node[anchor=east] {$\frac14$};

\draw (0,2.5) -- (1.25,0)
      (0,1.25)-- (0.833,0.833);

\draw[dashed,thick] (0.833,0.833) -- (2.5,0);

\draw[<-] (0.5,1.6) -- (1,2) node[anchor=north]{$~$};
\draw[<-]  (0.5,1) -- (1,1.3) node[anchor=north]{$~$};
\path (2,-1) node(caption){Figure 1: wave admissible pair $\Lambda_s^W$};

\end{tikzpicture}
\end{center}
\noindent
We have the following result.

\begin{theorem}\label{thm:Stri}
Let $a\in L^\infty(\mathbb{S}^{1},\mathbb{R}), {\A}\in
W^{1,\infty}(\mathbb{S}^{1},\mathbb{R}^2)$, assume \eqref{eq:transversal} and \eqref{equ:condassa},
and let $u$ be as in \eqref{eq:solution}.
Then there exists a constant $C$ such that
\begin{equation}\label{Str-w}
\|u\|_{L^q_t(\R;L^r(\R^2))}\leq C \big(\|f\|_{\dot{H}_{{\A},0}^s(\R^2)}+\|g\|_{\dot{H}_{{\A},0}^{s-1}(\R^2)}\big),
\end{equation}
for any $s\in\R$ and $(q,r)\in\Lambda_s^W$ given by \eqref{adm-p-w}.

\end{theorem}

\begin{remark}\label{rem:1}
Theorem \ref{thm:Stri} is a completely new result for the wave equation with a critical magnetic field. We stress that a time-decay estimate as \eqref{eq:decayshro} is not available, in this setting. We find particularly interesting the inequality \eqref{Str-w}, in the case of a negative inverse-square electric potential $a(\hat{x})\equiv a$, with
\begin{equation}\label{eq:question}
-\min_{k\in\Z}\{|k-\Phi_{\A}|\}^2<a<0,
\end{equation}
for which the role played by the magnetic potential is crucial. In analogy with the case of the Schr\"odinger equation, time-decay should not hold in the range \eqref{eq:question}, but this is an open question.
\end{remark}

\begin{remark}\label{rem:2}
Related to the above remark, another interesting open question is concerned with the critical inverse square potential
$$
a(\hat{x})\equiv-\min_{k\in\Z}\{|k-\Phi_{\A}|\}^2.
$$
In this case, in analogy with \cite{M}, one may ask about the validity of \eqref{Str-w}, and in particular of the endpoint estimate.
\end{remark}

\begin{remark}\label{rem:3}
Notice that the magnetic Sobolev norms $\dot H^s_{{\A},0}$ at the right-hand side of \eqref{Str-w} cannot be replaced by the usual Sobolev norms $\dot H^s$, since $\dot H^s(\R^2)\setminus\dot H^s_{{\A},0}(\R^2)\neq\emptyset$, for critical magnetic potentials, hence the evolution cannot be well defined on $\dot H^s$. On the other hand, in dimension $n\geq3$, the spaces $\dot H^s_{{\A},0}$ and $\dot H^s$ are equivalent (see \cite[cf. Lemma 2.3 - (i)]{FFT} for details), therefore one can wonder whether Strichartz estimates like
\begin{equation}\label{Str-w2}
\|u\|_{L^q_t(\R;L^r(\R^n))}\leq C \big(\|f\|_{\dot{H}^s(\R^n)}+\|g\|_{\dot{H}^{s-1}(\R^n)}\big)
\end{equation}
hold, in dimension $n\geq3$, for $n$-wave admissible pairs  
\begin{equation}\label{eq:dwaveda}
(q,r)\in\Lambda_s^W:=\big\{2\leq q,r\leq\infty, r\neq\infty, \tfrac2q+\tfrac {n-1}r\leq \tfrac {n-1}2, \;s=n\big(\tfrac12-\tfrac1r\big)-\tfrac1q\big\},\;s\in\R.
\end{equation}
The validity of \eqref{eq:question} for ${\A}\neq0$ is a completely open problem, at our knowledge.
\end{remark}
\noindent
The proof of Theorem \ref{thm:Stri} is inspired to the perturbation arguments in \cite{BPST, BPSS}. Nevertheless, to treat $\mathcal L_{{\A},a}$ as a perturbation of $-\Delta$, as in that case, involving local smoothing estimates, would require an estimate like
$$
\big\| |x|^{-\frac12}e^{it\sqrt{\LL_{{\A},a}}}f\big\|_{L_{t,x}^2(\R\times\R^2)}\leq \|f\|_{L^2(\R^2)},
$$
to handle the first-order term coming from the magnetic potential. Unfortunately, this estimate is known to be false, even in the free case, by the standard Agmon-H\"ormander Theory. To overcome this difficulty, we treat $\mathcal L_{{\A},a}$ as an electric perturbation of the purely magnetic operator $\mathcal L_{{\A},0}$. \vspace{0.2cm}

{\bf Acknowledgments:}\quad  The authors would like to thank Federico Cacciafesta for his
helpful discussions.
J. Zhang was supported by National Natural
Science Foundation of China (11771041,11831004,11671033) and H2020-MSCA-IF-2017(790623). J. Zheng was partially supported by the NSFC
under grants  11831004 and 11901041.
\vspace{0.2cm}

\section{Analysis of the operator $\LL_{{\A},a}$}
In this section, we study the harmonic analytical features of the operator $\LL_{{\A},a}$, relying on  the spectral properties proved in \cite{FFFP}.\vspace{0.2cm}

First of all, by using polar coordinates, thanks to condition \eqref{eq:transversal} we can write
\begin{equation}\label{LAa-r}
\begin{split}
\mathcal{L}_{{\A},a}=-\partial_r^2-\frac{1}r\partial_r+\frac{L_{{\A},a}}{r^2},
\end{split}
\end{equation}
where
\begin{equation}\label{L-angle}
\begin{split}
L_{{\A},a}&=(i\nabla_{\mathbb{S}^{1}}+{\A}(\hat{x}))^2+a(\hat{x})
\\&=-\Delta_{\mathbb{S}^{1}}+\big(|{\A}(\hat{x})|^2+a(\hat{x})+i\,\mathrm{div}_{\mathbb{S}^{1}}{\A}(\hat{x})\big)+2i {\A}(\hat{x})\cdot\nabla_{\mathbb{S}^{1}}.
\end{split}
\end{equation}
Let $\hat{x}=(\cos\theta,\sin\theta)$: then
\begin{equation*}
\partial_\theta=-\hat{x}_2\partial_{\hat{x}_1}+\hat{x}_1\partial_{\hat{x}_2},\quad \partial_\theta^2=\Delta_{\mathbb{S}^{1}}.
\end{equation*}
We define $A(\theta):[0,2\pi]\to \R$ as follows
\begin{equation}\label{equ:Athetadef}
  A(\theta)={\bf A}(\cos\theta,\sin\theta)\cdot (-\sin\theta,\cos\theta).
\end{equation}
Hence by \eqref{eq:transversal} we can write
\begin{equation}
{\bf A}(\cos\theta,\sin\theta)=A(\theta)(-\sin\theta,\cos\theta),\quad \theta\in[0,2\pi],
\end{equation}
therefore we obtain
\begin{equation}\label{LAa-s}
\begin{split}
L_{{\A},a}&=-\Delta_{\mathbb{S}^{1}}+\big(|{\A}(\hat{x})|^2+a(\hat{x})+i\,\mathrm{div}_{\mathbb{S}^{1}}{\A}(\hat{x})\big)+2i {\A}(\hat{x})\cdot\nabla_{\mathbb{S}^{1}}\\
&=-\partial_\theta^2+\big(|{A}(\theta)|^2+a(\theta)+i\,{A'}(\theta)\big)+2i {A}(\theta)\partial_\theta\\
&=: L_{{A},a}.
\end{split}
\end{equation}

\subsection{Spectral properties of  $\LL_{{\A},a}$} In this subsection, we consider the perturbation from the magnetic potential $\A$ and electrical potential $a$.\vspace{0.2cm}

First of all, recall \eqref{LAa-s}. The operator $L_{{\A},a}$ on $L^2(\mathbb{S}^1)$ has a compact inverse, hence by  classical Spectral Theory, its spectrum is purely discrete, and it is made by a countable family of real eigenvalues with finite multiplicity. We denote them with $\{\mu_k({\A},a)\}_{k=1}^\infty$, enumerated  such that
\begin{equation}\label{eig-Aa}
\mu_1({\A},a)\leq \mu_2({\A},a)\leq \cdots
\end{equation}
and we repeat each eigenvalue as many times as its multiplicity, and $\lim\limits_{k\to\infty}\mu_k({\A},a)=+\infty$ (see \cite[Lemma A.5]{FFT} for further details).
\begin{remark}\label{rem:new}
We remark that
\begin{equation}\label{equ:mu1}
  \mu_1(A,0)=\min_{k\in\Z}\{|k-\Phi_{\A}|^2\}
\end{equation}
(see \cite[(28)]{FFT}).
Notice that assumption \eqref{equ:condassa} implies that $\mu_1({\A},a)>0$.
\end{remark}
For each $k\in\mathbb{N}, k\geq1$, let $\psi_k(\hat{x})\in L^2(\mathbb{S}^{1})$ be the normalized eigenfunction of the operator $L_{{\A},a}$ corresponding to the $k$-th eigenvalue $\mu_k({\A},a)$, i.e. satisfying that
\begin{equation}\label{equ:eig-Aa}
\begin{cases}
L_{{\A},a}\psi_k(\hat{x})=\mu_k({\A},a)\psi_k(\hat{x})\quad \text{on}\,\quad  \mathbb{S}^{1};\\
\int_{\mathbb{S}^{1}}|\psi_k(\hat{x})|^2 d\hat{x}=1.
\end{cases}\end{equation}

For $f\in L^2(\R^2)$, based on the $\{\psi_k(\theta)\}_{k=0}^\infty$ where $\psi_0(\theta)=1/\sqrt{2\pi}$ and $\psi_k(\theta)$ given in \eqref{equ:eig-Aa}, we write $f$ into the form of separating variables
\begin{equation}\label{sep.v}
f(x)
=\sum_{k=0}^\infty c_{k}(r)\psi_{k}(\theta)
\end{equation}
where
\begin{equation*}
 c_{k}(r)=\int_{0}^{2\pi}f(r,\theta)
\psi_{k}(\theta) d\theta,
\end{equation*}
then
\begin{equation}\label{norm1}
\|f(r,\theta)\|^2_{L^2_{\theta}([0,2\pi])}=\sum_{k=0}^\infty|c_{k}(r)|^2.
\end{equation}
From the fact that
\begin{equation}\label{operator-t}
\mathcal{L}_{{\A},a}=-\partial^2_r-\frac{1}r\partial_r+\frac{L_{{\A},a}}{r^2},
\end{equation}
then, on each space $\mathcal{H}^{k}=\text{span}\{\psi_k\}$, the action of the operator is given by
\begin{equation*}
\begin{split}
\LL_{{\A},a}=-\partial_r^2-\frac{1}r\partial_r+\frac{\mu_k}{r^2}.
\end{split}
\end{equation*}
Let $\nu=\nu_k=\sqrt{\mu_k}$,  for $f\in L^2(\R^2)$, we define the Hankel transform of order $\nu$
\begin{equation}\label{hankel}
(\mathcal{H}_{\nu}f)(\rho,\theta)=\int_0^\infty J_{\nu}(r\rho)f(r,\theta) \,rdr,
\end{equation}
where the Bessel function of order $\nu$ is given by
\begin{equation}\label{Bessel}
J_{\nu}(r)=\frac{(r/2)^{\nu}}{\Gamma\left(\nu+\frac12\right)\Gamma(1/2)}\int_{-1}^{1}e^{isr}(1-s^2)^{(2\nu-1)/2} ds, \quad \nu>-1/2, r>0.
\end{equation}

\begin{lemma}\label{lem: J} Let $J_\nu(r)$ be the Bessel function defined in \eqref{Bessel} and $R\gg1$, then there exists a constant $C$ independent of $\nu$ and $R$ such that
\begin{equation}\label{bessel-r}
|J_\nu(r)|\leq
\frac{Cr^\nu}{2^\nu\Gamma(\nu+\frac12)\Gamma(1/2)}\left(1+\frac1{\nu+1/2}\right),
\end{equation}
and
\begin{equation}\label{est:b}
\int_{R}^{2R} |J_\nu(r)|^2 dr \leq C.
\end{equation}
\end{lemma}

\begin{proof} The first one is obtained by a direct computation. The  inequality \eqref{est:b} is  a direct consequence of the asymptotically behavior of Bessel function; see \cite[Lemma 2.2]{MZZ}.

\end{proof}

The Hankel transform satisfies the following properties (see \cite{BPSS, PST}):
\begin{lemma}\label{Lem:Hankel}
Let $\mathcal{H}_{\nu}$ be defined as above  and $A_{\nu}:=-\partial_r^2-\frac{1}r\partial_r+\frac{\nu^2}{r^2}$. Then

$(\rm{i})$ $\mathcal{H}_{\nu}=\mathcal{H}_{\nu}^{-1}$,

$(\rm{ii})$ $\mathcal{H}_{\nu}$ is self-adjoint, i.e.
$\mathcal{H}_{\nu}=\mathcal{H}_{\nu}^*$,

$(\rm{iii})$ $\mathcal{H}_{\nu}$ is an $L^2$ isometry, i.e.
$\|\mathcal{H}_{\nu}\phi\|_{L^2_\xi}=\|\phi\|_{L^2_x}$,

$(\rm{iv})$ $\mathcal{H}_{\nu}(
A_{\nu}\phi)(\xi)=|\xi|^2(\mathcal{H}_{\nu} \phi)(\xi)$, for
$\phi\in L^2$.
\end{lemma}

Briefly recalling the functional calculus for well-behaved functions $F$ (see \cite{Taylor}),
\begin{equation}\label{funct}
\begin{split}
F(\mathcal{L}_{{\A},a}) f(r,\theta)&=\sum_{k=0}^\infty \psi_{k}(\theta) \int_0^\infty F(\rho^2) J_{\nu_k}(r\rho)b_{k}(\rho) \,\rho d\rho\\
&= \int_0^\infty\int_0^{2\pi} K(r,\theta,r_2,\theta_2) f(r_2,\theta_2) r_2 dr_2 d\theta_2
\end{split}
\end{equation}
where $b_{k}(\rho)=(\mathcal{H}_{\nu_k}a_k)(\rho)$, $f(r,\theta)=\sum\limits_{k=0}^\infty a_{k}(r)\psi_{k}(\theta)$ and
\begin{equation}\label{gkernel}
\begin{split}
K(r,\theta,r_2,\theta_2)=\sum_{k=0}^\infty \psi_{k}(\theta)\overline{ \psi_{k}(\theta_2)} \int_0^\infty F(\rho^2) J_{\nu_k}(r\rho) J_{\nu_k}(r_2\rho)\,\rho d\rho .
\end{split}
\end{equation}

\subsection{Sobolev embedding}

\begin{lemma}\label{Lem:sobequil2}
Let $a, {\A}\in
W^{1,\infty}(\mathbb{S}^{1},\mathbb{R}^2)$, and assume \eqref{equ:condassa}. Then
\begin{equation}\label{equ:equisobeqi}
  \|f\|_{\dot{H}^s_{{\A},0}(\R^2)}\simeq \|f\|_{\dot{H}^s_{{\A},a}(\R^2)},
\end{equation}
for all $s\in[-1,1]$.
\end{lemma}

\begin{proof}
For $s=1$, the proof is an immediate consequence of assumption \eqref{equ:ghardyLW}.
Indeed, we have
\begin{equation}
 \|f\|_{\dot{H}^1_{{\A},a}(\R^2)}^2=\int_{\R^2} \big(|\nabla_{\A}f|^2+\frac{a(\hat{x})}{|x|^2}|f|^2\big)dx
\end{equation}
then
\begin{equation}
\int_{\R^2} \big(|\nabla_{\A}f|^2-\frac{a_-(\hat{x})}{|x|^2}|f|^2\big)dx\leq \|f\|_{\dot{H}^1_{{\A},a}(\R^2)}^2\leq \int_{\R^2} \big(|\nabla_{\A}f|^2+\frac{a_+(\hat{x})}{|x|^2}|f|^2\big)dx
\end{equation}
where $a_-:=\max\{0,-a\}$ and $a_+:=\max\{0,a\}$. From \eqref{equ:condassa} and \eqref{equ:ghardyLW},  there exist a small constant $c$ and a large constant $C$ such that
\begin{equation}
\begin{split}
\int_{\R^2} \big(|\nabla_{\A}f|^2-\frac{a_-(\hat{x})}{|x|^2}|f|^2\big)dx&\geq \big(1-\frac{\|a_-\|_{L^\infty(\mathbb{S}^1)}}{ \min_{k\in\Z}\{|k-\Phi_{\A}|\}^2}\big) \int_{\R^2} |\nabla_{\A}f|^2dx
\\& \geq c\int_{\R^2} |\nabla_{\A}f|^2dx
 \end{split}
\end{equation}
and
\begin{equation}
\begin{split}
 \int_{\R^2} \big(|\nabla_{\A}f|^2+\frac{a_+(\hat{x})}{|x|^2}|f|^2\big)dx&\leq \big(1+\frac{\|a_+\|_{L^\infty(\mathbb{S}^1)}}{ \min_{k\in\Z}\{|k-\Phi_{\A}|\}^2}\big) \int_{\R^2} |\nabla_{\A}f|^2dx
\\& \leq C\int_{\R^2} |\nabla_{\A}f|^2dx.
 \end{split}
\end{equation}
Then, by duality and interpolation, one obtains the full range $s\in[-1,1]$.
\end{proof}

Finally, we derive the following Sobolev embedding.

\begin{lemma}\label{cor:sobl2}
Let $a, {\A}\in
W^{1,\infty}(\mathbb{S}^{1},\mathbb{R}^2)$, and assume \eqref{equ:condassa}. Then,
\begin{equation}\label{equ:sobeml2}
  \|f\|_{L^p(\R^2)}\leq C\|f\|_{\dot{H}^{1-\frac2p}_{{\A},a}(\R^2)}
\end{equation}
for any $2\leq p<+\infty$.
\end{lemma}

\begin{proof}
 In the purely magnetic case $a\equiv0$, this immediately follows by the usual Sobolev embedding $\dot H^{1-\frac2p}(\R^2)\hookrightarrow L^p(\R^2)$, and
 the diamagnetic inequality (see \cite[Lemma A.1]{FFT})
 $$
|\nabla|f|(x)| \leq |\nabla_{\A} f(x)|.
 $$
 Then the proof follows by Lemma \ref{Lem:sobequil2}.
\end{proof}

\subsection{Littlewood-Paley theory for $\LL_{{\A},0}$}

In this subsection, we establish some harmonic analysis tools associated to the purely magnetic operator $\LL_{{\A},0}$. Our results in this subsection rely on the  heat kernel estimate
\begin{equation}\label{equ:heat}
\big|e^{-t\LL_{{\A},0}}(x,y)\big|\lesssim
\frac{1}{t}\exp\Big(-\frac{|x-y|^2}{4t}\Big),\quad \forall\;t>0,
\end{equation}
which will be given by Proposition \ref{prop:pointest} in next section.
Then
\begin{equation}\label{equ:Mihmult}
  m(\sqrt{\LL_{{\A},0}}):L^p(\R^2)\mapsto L^p(\R^2),\quad 1<p<\infty
\end{equation}
where $m\in C^N(\R)$ satisfies the weaker
Mikhlin-type condition for $N\geq2$
\begin{equation}\label{Minhlin}
\sup_{0\leq k\leq
N}\sup_{\lambda\in\R}\Big|\big(\lambda\partial_\lambda\big)^k
m(\lambda)\Big|\leq C<\infty.
\end{equation}
This result that the Gaussian boundedness of heat kernel implies the boundedness of Mikhlin-type multiplier  is standard, for example, see Alexopoulos \cite[Theorem 6.1]{Alex} or \cite{KMVZZ}.

Based on \eqref{equ:Mihmult}, we have the Littlewood-Paley theory result for $\LL_{{\A},0}$. We follow the statements in \cite{KMVZZ}.
Let $\phi:[0,\infty)\to[0,1]$ be a smooth function such that
\begin{align*}
\phi(\lambda)=1 \qtq{for} 0\le\lambda\le 1 \qtq{and} \phi(\lambda)=0\qtq{for} \lambda\ge 2.
\end{align*}
For each integer $j\in \Z$, we define
\begin{align*}
\phi_j(\lambda):=\phi(\lambda/2^j) \qtq{and} \varphi_j(\lambda):=\phi_j(\lambda)-\phi_{j-1}(\lambda).
\end{align*}
Clearly, $\{\varphi_j(\lambda)\}_{j\in \Z} $ forms a partition of unity for $\lambda\in(0,\infty)$.  We define the Littlewood-Paley projections as follows:
\begin{align}\label{equ:lidef}
P_{\le j} :=\phi_j\bigl(\sqrt{\LL_{{\A},0}}\,\bigr), \quad P_j :=\varphi_j(\sqrt{\LL_{{\A},0}}\,\bigr), \qtq{and} P_{>j} :=I-P_{\le j}.
\end{align}

We also define another family of Littlewood-Paley projections via the heat kernel, as follows:
\begin{align*}
\tilde{P}_{\le j} := e^{-\LL_{{\A},0}/2^{2j}}, \quad \tilde{P}_j:=e^{-\LL_{{\A},0}/2^{2j}}-e^{-4\LL_{{\A},0}/2^{2j}}, \qtq{and} \tilde{P}_{>j}:=I-\tilde{P}_{\le j}.
\end{align*}

\begin{proposition}[Bernstein inequality] Let $\{\varphi_j\}_{j\in\mathbb Z}$ be a Littlewood-Paley sequence, let ${\A}\in W^{1,\infty}(\mathbb S^1)$ and assume \eqref{eq:transversal}.
Then for $1< p\leq q<\infty$ or $1\leq p<q\leq \infty$,
there exists constant $C_{p,q}$ depending on $p,q$ such that
\begin{equation}\label{bern1}
\Big\|\varphi_j(\sqrt{\LL_{{\A},0}})f\Big\|_{L^q(\R^2)}\leq
C_{p,q}2^{2j(\frac1p-\frac1q)}\|f\|_{L^p(\R^2)},
\end{equation}
and for $1\leq p\leq \infty$
\begin{equation}\label{bern2}
\Big\|\nabla_{\A}\varphi_j(\sqrt{\LL_{{\A},0}})f\Big\|_{L^p(\R^2)}\leq
C_{p,q}2^{j}\|f\|_{L^p(\R^2)}.
\end{equation}

\end{proposition}

\begin{proof} We first prove \eqref{bern1}. By scaling, it suffices to show \eqref{bern1} with $j=0$.  As  $P_j$  can be written as products of $\tilde{P}_{\leq j}$ with $L^p$-bounded multipliers by \eqref{equ:Mihmult}, it suffices to prove that  $\tilde{P}_{\leq j}$ is bounded from $L^p$ to $L^q$. This just follows from the heat kernel estimate \eqref{equ:heat} and Young's inequality for $1< p\leq q<\infty$ or $1\leq p<q\leq \infty$.

Next we prove \eqref{bern2}. By interpolation and dual argument, it suffices to prove \eqref{bern2} with $p=1$ and $p=2$. We
follow the arguments of \cite{CD, IO} based on the heat kernel Proposition \ref{prop:heatkernel}. We sketch it as follows.
When $p=2$, it follows from the standard functional calculus for the self-adjoint operator $\LL_{{\A},0}$.
To estimate \eqref{bern2} for $p=1$, we set
$$p_t(x,y):=e^{-t\LL_ {{\A},0}}(x,y).$$
Then, by Proposition \ref{prop:pointest} and Proposition \ref{prop:heatkernel}, we get
\begin{align}\label{equ:pointheat}
  |p_t(x,y)|\lesssim & t^{-1}e^{-\frac{|x-y|^2}{4t}},\\\label{equ:pointheatder}
  \big|\tfrac{\pa}{\pa t}p_t(x,y)\big|+\big|\LL_{{\A},0}p_t(x,y)\big|\lesssim & t^{-2}e^{-\frac{|x-y|^2}{8t}}.
\end{align}
Then for $\gamma<\tfrac14$, by using argument of Grigor'yan\cite{Grigo} or \cite[Lemma 2.3]{CD}, there holds
\begin{equation}\label{equ:heatinetq}
  \int_{\R^2}\big|\nabla_{\A}p_t(x,y)\big|^2e^{\gamma\frac{|x-y|^2}{t}}\;dx\leq C_\gamma t^{-2},\quad\forall\;y\in\R^2,t>0.
\end{equation}
As a consequence of Cauchy-Schwartz's inequality, we obtain
\begin{equation}\label{equ:heatintl1}
   \int_{\R^2}\big|\nabla_{\A}p_t(x,y)\big|\;dx\lesssim t^{-\frac12},\quad\forall\;y\in\R^2,t>0.
\end{equation}
Hence we get
\begin{equation}\label{equ:heatkernelderest}
  \big\|\nabla_{\A}e^{-\LL_{{\A},0}}f\big\|_{L^1(\R^2)}\leq C\|f\|_{L^1(\R^2)}.
\end{equation}
Let $\varphi$ be as above, and set $\psi(x)=\varphi(\sqrt{x})$ and $\psi_e(x):=\psi(x)e^{2x}$. We can extend $\psi$  as a $C^\infty_c$-function on $\R$ and
then its Fourier
transform $\hat{\psi}_e$ belongs to Schwartz class.
From the argument of \cite[Theorem 2.1]{IO}, we prove
\begin{align*}
 &\Big\|\nabla_{\A}\varphi(\sqrt{\LL_{{\A},0}})f\Big\|_{L^1(\R^2)}\\
 \lesssim & \int_{\R} \big| \hat{\psi}_e(\xi)\big|\cdot\big\| e^{-(1-i\xi)\LL_ {{\A},0}}\big\|_{L^1\to L^1}\big\|\nabla_{\A}e^{-\LL_ {{\A},0}}f\big\|_{L^1}\;d\xi\lesssim \|f\|_{L^1}.
\end{align*}


\end{proof}

\begin{proposition}[L-P square function inequality]\label{prop:squarefun} Let $\{\varphi_j\}_{j\in\mathbb Z}$ be a Littlewood-Paley sequence, let ${\A}\in W^{1,\infty}(\mathbb S^1)$ and assume \eqref{eq:transversal}.
Then for $1<p<\infty$,
there exist constants $c_p$ and $C_p$ depending on $p$ such that
\begin{equation}\label{square}
c_p\|f\|_{L^p(\R^2)}\leq
\Big\|\Big(\sum_{j\in\Z}|\varphi_j(\sqrt{\LL_{{\A},0}})f|^2\Big)^{\frac12}\Big\|_{L^p(\R^2)}\leq
C_p\|f\|_{L^p(\R^2)}.
\end{equation}

\end{proposition}

\begin{proof}
Based on \eqref{equ:Mihmult}, Proposition \ref{prop:squarefun} follows from
the Rademacher functions argument in \cite{Stein}.
We omit the details here and refer the reader to \cite{Zhang} for details.
\end{proof}

\section{Heat kernel estimates}
In this section, we construct a representation of the heat kernel and prove  the Gaussian boundedness \eqref{equ:heat}.
The study of heat kernel for the Schr\"odinger operator with Aharonov-Bohm potential has independent interest, see \cite{Ko, MOR} and reference therein.  

\begin{proposition}[Heat kernel]\label{prop:heatkernel} Let $x=r_1(\cos\theta_1,\sin\theta_1)$ and $y=r_2(\cos\theta_2,\sin\theta_2)$,
then we have the expression of heat kernel
\begin{align}\label{equ:heatkernel}
  &e^{-t\LL_{{\A},0}}(x,y)\\\nonumber
  =&\frac{e^{-\frac{|x-y|^2}{4t}} }{t}\frac{e^{i\int_{\theta_1}^{\theta_2}A(\theta')d\theta'}}{2\pi}\big(\mathbbm{1}_{[0,\pi]}(|\theta_1-\theta_2|)
  +e^{-i2\pi\alpha}\mathbbm{1}_{[\pi,2\pi]}(|\theta_1-\theta_2|)\big)\\\nonumber
  &-\frac{1}{\pi}\frac{e^{-\frac{r_1^2+r_2^2}{4t}} }{t} e^{-i\alpha(\theta_1-\theta_2)+i\int_{\theta_2}^{\theta_{1}}A(\theta') d\theta'} \int_0^\infty e^{-\frac{r_1r_2}{2t}\cosh s} \Big(\sin(|\alpha|\pi)e^{-|\alpha|s}\\\nonumber
  &\qquad +\sin(\alpha\pi)\frac{(e^{-s}-\cos(\theta_1-\theta_2+\pi))\sinh(\alpha s)-i\sin(\theta_1-\theta_2+\pi)\cosh(\alpha s)}{\cosh(s)-\cos(\theta_1-\theta_2+\pi)}\Big) ds\\\nonumber
  \triangleq&G_h(r_1,\theta_1,r_2,\theta_2)+D_h(r_1,\theta_1,r_2,\theta_2).
\end{align}
\end{proposition}

\begin{remark} In particular $\alpha=0$, $D_h(r_1,\theta_1,r_2,\theta_2)$ vanishes and $G_h(r_1,\theta_1,r_2,\theta_2)$ is the same to
the classical heat kernel representation in Euclidean space without potential.
\end{remark}

\begin{proof}
From \eqref{LAa-s} with $a(\theta)=0$, then $L_{{\A},0}=(i\partial_\theta+A(\theta))^2$. For simple, as \eqref{equ:defphia1}, let
\begin{equation}
 \alpha= \Phi_{\A}=\frac{1}{2\pi}\int_0^{2\pi}A(\theta)\;d\theta.
\end{equation}
Due to  unitarily equivalent of magnetic Schr\"odinger operators for $\alpha$ and $\alpha+1$, without loss of generality, we assume $\alpha\in(-1,0)\cup (0,1)$.\vspace{0.2cm}

Following \cite{LW} of Laptev-Weidl, the operator $i\partial_\theta+A(\theta)$ with domain $H^1(\mathbb{S}^1)$ in $L^2(\mathbb{S}^1)$ has eigenvalue
$\nu(k)=k+\alpha, k\in\Z$ and eigenfunction
\begin{equation}\label{eig-f}
\varphi_k(\theta)=\frac1{\sqrt{2\pi}}e^{-i\big(\theta(k+\alpha)-\int_0^{\theta}A(\theta') d\theta'\big)}.
\end{equation}
Then
$$L_{{\A},0}\varphi_k(\theta)=(k+\alpha)^2\varphi_k(\theta).$$
Let $\nu=\nu_k=|k+\alpha|$ with $k\in\Z$.
Briefly recalling the functional calculus for well-behaved functions $F$ (see \cite{Taylor}),
\begin{equation}\label{funct1}
F(\mathcal{L}_{{\A},0}) f(r,\theta)=\sum_{k\in\Z} \varphi_{k}(\theta) \int_0^\infty F(\rho^2) J_{\nu_k}(r\rho)b_{k}(\rho) \,\rho d\rho
\end{equation}
where $b_{k}(\rho)=(\mathcal{H}_{\nu_k}a_k)(\rho)$ and $f(r,\theta)=\sum\limits_{k\in\Z}a_{k}(r)\varphi_{k}(\theta)$. 
Thus, the kernel of the operator  $e^{-t\LL_{{\A},0}}$ is given by
\begin{equation}\label{equ:kerneita}
K(t,x,y)=K(t,r_1,\theta_1,r_2,\theta_2)=\sum_{k\in\Z}\varphi_{k}(\theta_1)\overline{\varphi_{k}(\theta_2)}K_{\nu_k}(t,r_1,r_2).
\end{equation}
where $K_{\nu}$ is given by 
\begin{equation}\label{equ:knukdef12}
  K_{\nu}(t,r_1,r_2)=\int_0^\infty e^{-t\rho^2}J_{\nu}(r_1\rho)J_{\nu}(r_2\rho) \,\rho d\rho=\frac{e^{-\frac{r_1^2+r_2^2}{4t}}}{t} I_\nu\big(\frac{r_1r_2}{2t}\big)
\end{equation}
where $I_\nu$ is the modified Bessel function and we use the Weber identity, e.g. \cite[Proposition 8.7]{Taylor}.
For the modified Bessel function $I_\nu$,   we use the integral representation in \cite{Watson}  to write: for $z=\frac{r_1r_2}{2t}$
\begin{equation}\label{m-bessel}
I_\nu(z)=\frac1{\pi}\int_0^\pi e^{z\cos(s)} \cos(\nu s) ds-\frac{\sin(\nu\pi)}{\pi}\int_0^\infty e^{-z\cosh s} e^{-s\nu} ds
\end{equation}
By \eqref{eig-f}, we need to consider
 \begin{equation}\label{I}
 \begin{split}&\sum_{k\in\Z} e^{-i\big((\theta_1-\theta_2)(k+\alpha)-\int_{\theta_2}^{\theta_{1}}A(\theta') d\theta'\big)} I_{\nu_k}\big(\frac{r_1r_2}{2t}\big), \\
=& e^{-i\alpha(\theta_1-\theta_2)+i\int_{\theta_2}^{\theta_{1}}A(\theta') d\theta'} \sum_{k\in\Z} e^{-ik(\theta_1-\theta_2)}  I_{\nu_k}\big(\frac{r_1r_2}{2t}\big).
 \end{split}
 \end{equation}
Recall $\nu_k=|k+\alpha|, k\in\Z$, note that on the line, we have that
\begin{equation*}
\begin{split}
\sum_{k\in\Z} \cos(s\nu) e^{-ik(\theta_1-\theta_2)}&=\sum_{k\in\Z} \frac{e^{i(k+\alpha)s}+e^{-i(k+\alpha)s}}2 e^{-ik(\theta_1-\theta_2)}\\
&=\frac12\big(e^{-i\alpha s}\delta(\theta_1-\theta_2+s)+e^{i\alpha s}\delta(\theta_1-\theta_2-s)\big).
\end{split}
\end{equation*}
 To get $\cos(s\sqrt{L_{{\A},0}})$ on $\R/(2\pi\Z)$,  by the method of image in \cite{Taylor}, we make the above periodic
 \begin{equation}
 \begin{split}
 &\cos(s\sqrt{L_{{\A},0}})\delta(\theta_1-\theta_2)\\
 =&
 \frac12\sum_{j\in\Z}\big[e^{-is\alpha}\delta(\theta_1+2j\pi-\theta_2+s)+e^{is\alpha}\delta(\theta_1+2j\pi-\theta_2-s)\big].
 \end{split}
 \end{equation}
To consider \eqref{I}, we consider the first term in $I_\nu(z)$ to obtain
\begin{align*}
&\frac1{\pi}\sum_{k\in\Z} e^{-ik(\theta_1-\theta_2)}\int_0^\pi e^{z\cos(s)} \cos(\nu_k s) ds\\\nonumber
 =&\frac1{\pi}\int_0^{\pi} e^{z\cos(s)} \cos(s\sqrt{L_{{\A},0}})\delta(\theta_1-\theta_2)\;ds\\\nonumber
 =&\frac1{2\pi}\sum_{j\in\Z}\int_0^{\pi} e^{z\cos(s)}\big[e^{-is\alpha}\delta(\theta_1-\theta_2+2j\pi+s)
 +e^{is\alpha}\delta(\theta_1-\theta_2+2j\pi-s)\big]\;ds\\\nonumber
 =&\frac1{\pi}\sum_{\{j\in\Z: 0\leq |\theta_1-\theta_2+2j\pi|\leq \pi\}} e^{z\cos(\theta_1-\theta_2+2j\pi)}e^{i(\theta_1-\theta_2+2j\pi)\alpha}
\\
 =&\frac1{\pi}\times\begin{cases}
e^{z\cos(\theta_1-\theta_2)} e^{i(\theta_1-\theta_2)\alpha}\quad&\text{if}\quad |\theta_1-\theta_2|<\pi\\
e^{z\cos(\theta_1-\theta_2)}e^{i(\theta_1-\theta_2-2\pi)\alpha}\quad&\text{if}\quad \pi<|\theta_1-\theta_2|<2\pi.
 \end{cases}
\end{align*}
Therefore the contribution of the first term is
\begin{equation}
\frac{e^{-\frac{r_1^2+r_2^2}{4t}} }{t}\frac{e^{i\int_{\theta_1}^{\theta_2}A(\theta')d\theta'}}{2\pi}e^{\frac{r_1r_2}{2t}\cos(\theta_1-\theta_2)} \big(\mathbbm{1}_{[0,\pi]}(|\theta_1-\theta_2|)+e^{-i2\pi\alpha}\mathbbm{1}_{[\pi,2\pi]}(|\theta_1-\theta_2|)\big)
\end{equation}
which gives the term $G_h(r_1,\theta_1,r_2,\theta_2)$ in \eqref{equ:heatkernel}.

We now consider the second term associated with \eqref{m-bessel}
 \begin{align}
\nonumber&\sum_{k\in\Z} e^{-ik(\theta_1-\theta_2)}\frac{\sin(\nu\pi)}{\pi}\int_0^\infty e^{-z\cosh s} e^{-s\nu} ds.
\end{align}
Recall $\nu_k=|k+\alpha|, k\in\Z$ and $\alpha\in(-1,1)\setminus\{0\}$,
then
\begin{equation}
\nu=|k+\alpha|=
\begin{cases}
k+\alpha,\qquad &k\geq1;\\
|\alpha|,\qquad &k=0;\\
-(k+\alpha),\qquad &k\leq -1.
\end{cases}
\end{equation}
Therefore we obtain
\begin{equation}
\begin{split}
\sin(\pi\nu_k)=\sin(\pi|k+\alpha|)=\begin{cases}
\cos(k\pi)\sin(\alpha\pi),\qquad &k\geq1;\\
\sin(|\alpha|\pi),\qquad &k=0;\\
-\cos(k\pi)\sin(\alpha\pi),\qquad &k\leq -1.
\end{cases}
\end{split}
\end{equation}
Therefore we furthermore have
\begin{equation}
\begin{split}
&\sum_{k\in\Z}\sin(\pi|k+\alpha|)e^{-s|k+\alpha|}e^{-ik(\theta_1-\theta_2)}\\
=&\sin(\alpha\pi)\sum_{k\geq 1} \frac{e^{ik\pi}+e^{-ik\pi}}2 e^{-s(k+\alpha)}e^{-ik(\theta_1-\theta_2)}+\sin(|\alpha|\pi)e^{-|\alpha|s}\\&-\sin(\alpha\pi)\sum_{k\leq-1} \frac{e^{ik\pi}+e^{-ik\pi}}2 e^{s(k+\alpha)}e^{-ik(\theta_1-\theta_2)}\\
=&\sin(|\alpha|\pi)e^{-|\alpha|s}+\frac{\sin(\alpha\pi)}2\Big(e^{-s\alpha}\sum_{k\geq1} e^{-ks} \big(e^{-ik(\theta_1-\theta_2+\pi)} +e^{-ik(\theta_1-\theta_2-\pi)}\big)\\&\qquad\qquad\qquad- e^{s\alpha} \sum_{k\geq 1} e^{-ks} \big(e^{ik(\theta_1-\theta_2+\pi)} +e^{ik(\theta_1-\theta_2-\pi)}\big)\Big).
\end{split}
\end{equation}
Note that
\begin{equation}
\sum_{k=1}^\infty e^{ikz}=\frac{e^{iz}}{1-e^{iz}},\qquad \mathrm{Im} z>0,
\end{equation}
we finally obtain
\begin{align}
&\sum_{k\in\Z}\sin(\pi|k+\alpha|)e^{-s|k+\alpha|}e^{-ik(\theta_1-\theta_2)}\\\nonumber
=&\sin(|\alpha|\pi)e^{-|\alpha|s}+\frac{\sin(\alpha\pi)}2\Big(\frac{e^{-(1+\alpha)s-i(\theta_1-\theta_2+\pi)}}
{1-e^{-s-i(\theta_1-\theta_2+\pi)}}+\frac{e^{-(1+\alpha)s-i(\theta_1-\theta_2-\pi)}}{1-e^{-s-i(\theta_1-\theta_2-\pi)}}\\\nonumber
&\qquad\qquad\qquad-\frac{e^{-(1-\alpha)s+i(\theta_1-\theta_2+\pi)}}{1-e^{-s+i(\theta_1-\theta_2+\pi)}}-
\frac{e^{-(1-\alpha)s+i(\theta_1-\theta_2-\pi)}}{1-e^{-s+i(\theta_1-\theta_2-\pi)}}\Big)\\\nonumber
=&\sin(|\alpha|\pi)e^{-|\alpha| s}+\sin(\alpha\pi)\frac{(e^{- s}-\cos(\theta_1-\theta_2+\pi))\sinh(\alpha s)
-i\sin(\theta_1-\theta_2+\pi)\cosh(\alpha s)}{\cosh( s)-\cos(\theta_1-\theta_2+\pi)}.
\end{align}
Therefore we obtain the contribution of the second term
\begin{equation}
\begin{split}
&-\frac{1}{\pi}\frac{e^{-\frac{r_1^2+r_2^2}{4t}} }{t}
e^{-i\alpha(\theta_1-\theta_2)+i\int_{\theta_2}^{\theta_{1}}A(\theta') d\theta'} \int_0^\infty e^{-\frac{r_1r_2}{2t}\cosh s} \Big(\sin(|\alpha|\pi)e^{-|\alpha|s} \\&+\sin(\alpha\pi)\frac{(e^{-s}-\cos(\theta_1-\theta_2+\pi))\sinh(\alpha s)-i\sin(\theta_1-\theta_2+\pi)\cosh(\alpha s)}{\cosh(s)-\cos(\theta_1-\theta_2+\pi)}\Big) ds.
\end{split}
\end{equation}

\end{proof}

\begin{proposition}[Pointwise estimate]\label{prop:pointest}
There hold
\begin{align}\label{equ:pcontr}
 |e^{-t\LL_{{\A},0}}(x,y)|\lesssim & t^{-1}e^{-\frac{|x-y|^2}{4t}}\quad \forall\;t>0,
\end{align}
and
\begin{align}\label{equ:pcontrder}
 \big|\tfrac{\pa}{\pa t}e^{-t\LL_{{\A},0}}(x,y)\big|\lesssim & t^{-2}e^{-\frac{|x-y|^2}{8t}}\quad \forall\;t>0.
\end{align}

\end{proposition}

\begin{remark} The first inequality \eqref{equ:pcontr}, which is non-trivial since the components of the Aharonov-Bohm potential do not belong to $L^2_{loc}(\R^2)$, was first proved in \cite{MOR}. We provide an alternative and direct proof based on our representation formula
in Proposition \ref{prop:heatkernel}. 
To our best knowledge, the second one \eqref{equ:pcontrder} is new.
\end{remark}

\begin{proof} We only prove \eqref{equ:pcontr} since \eqref{equ:pcontrder} follows from the same argument and the fact
$$xe^{-2x}\leq e^{-x},\quad \forall\;x\geq0.$$
By Proposition \ref{prop:heatkernel}, we have
  $$e^{-t\LL_{{\A},0}}(x,y)=G_h(r_1,\theta_1,r_2,\theta_2)+D_h(r_1,\theta_1,r_2,\theta_2),$$
and
\begin{equation}\label{equ:gcontes}
  |G_h(r_1,\theta_1,r_2,\theta_2)|\lesssim t^{-1}e^{-\frac{|x-y|^2}{4t}},
\end{equation}
and
\begin{align*}
  &|D_h(r_1,\theta_1,r_2,\theta_2)|\\
  \lesssim&\frac{e^{-\frac{r_1^2+r_2^2}{4t}}}{t} \Big|\int_0^\infty e^{-\frac{r_1r_2}{2t}\cosh s} \Big(\sin(|\alpha|\pi)e^{-|\alpha|s}\\\nonumber
  &\qquad +\sin(\alpha\pi)\frac{(e^{-s}-\cos(\theta_1-\theta_2+\pi))\sinh(\alpha s)-i\sin(\theta_1-\theta_2+\pi)\cosh(\alpha s)}{\cosh(s)-\cos(\theta_1-\theta_2+\pi)}\Big) ds\Big|\\
  \lesssim&\frac{e^{-\frac{(r_1+r_2)^2}{4t}}}{t} \Big|\int_0^\infty \Big(|\sin(|\alpha|\pi)|e^{-|\alpha|s}\\\nonumber
  &\qquad +|\sin(\alpha\pi)|\cdot\Big|\frac{(e^{-s}-\cos(\theta_1-\theta_2+\pi))\sinh(\alpha s)-i\sin(\theta_1-\theta_2+\pi)\cosh(\alpha s)}{\cosh(s)-\cos(\theta_1-\theta_2+\pi)}\Big|\Big) ds\Big|.
\end{align*}
  Note that $D(r_1,\theta_1,r_2,\theta_2)=0$ when $\alpha=0$. Thus, we only need to show that for $\alpha\in(-1,1)\backslash\{0\}$
  \begin{align}\label{equ:ream1}
    \int_0^\infty e^{-|\alpha|s}\;ds\lesssim&1\\\label{equ:ream2}
    \int_0^\infty  \Big|\frac{(e^{-s}-\cos(\theta_1-\theta_2+\pi))\sinh(\alpha s)}{\cosh(s)-\cos(\theta_1-\theta_2+\pi)}\Big|\;ds\lesssim&1
    \\\label{equ:ream3}
     \int_0^\infty  \Big|\frac{\sin(\theta_1-\theta_2+\pi)\cosh(\alpha s)}{\cosh(s)-\cos(\theta_1-\theta_2+\pi)}\Big|\;ds\lesssim&1.
  \end{align}
  It is easy to check \eqref{equ:ream1}.\vspace{0.2cm}

  {\bf Estimate of \eqref{equ:ream2}:}  Note that
$$\cosh(\tau)
-\cos(\theta_1-\theta_2+\pi)=\sinh^2\big(\tfrac{\tau}2\big)+\sin^2\big(\tfrac{\theta_1-\theta_2+\pi}2\big),$$
  we get
 \begin{align*}
    & \int_0^\infty  \Big|\frac{(e^{-s}-\cos(\theta_1-\theta_2+\pi))\sinh(\alpha s)}{\cosh(s)-\cos(\theta_1-\theta_2+\pi)}\Big|\;ds\\
    =&\int_0^1 \Big|\frac{(e^{-s}-1+1-\cos(\theta_1-\theta_2+\pi))\sinh(\alpha s)}{\sinh^2\big(\tfrac{s}2\big)+\sin^2\big(\tfrac{\theta_1-\theta_2+\pi}2\big)}\Big|\;ds
     \\&\qquad +\int_1^\infty \Big|\frac{(e^{-s}-\cos(\theta_1-\theta_2+\pi))\sinh(\alpha s)}{\sinh^2\big(\tfrac{s}2\big)+\sin^2\big(\tfrac{\theta_1-\theta_2+\pi}2\big)}\Big|\;ds\\
    \lesssim&\int_0^1\frac{(s+\tfrac{\theta_1-\theta_2+\pi}{2})s}{s^2+(\tfrac{\theta_1-\theta_2+\pi}{2})^2}\;ds
    +\int_1^\infty e^{-(1-\alpha)s}\;ds\\
    \lesssim&1.
 \end{align*}

  {\bf Estimate of \eqref{equ:ream3}:} we have
  \begin{align*}
     &  \int_0^\infty  \Big|\frac{\sin(\theta_1-\theta_2+\pi)\cosh(\alpha s)}{\cosh(s)-\cos(\theta_1-\theta_2+\pi)}\Big|\;ds\\
    \lesssim& \int_0^1  \Big|\frac{\sin^2\big(\tfrac{\theta_1-\theta_2+\pi}2\big)}{\sinh^2\big(\tfrac{s}2\big)+\sin^2\big(\tfrac{\theta_1-\theta_2+\pi}2\big)}\Big|\;ds
    +\int_1^\infty \frac{\cosh(\alpha s)}{\sinh^2\big(\tfrac{s}2\big)+\sin^2\big(\tfrac{\theta_1-\theta_2+\pi}2\big)}\;ds\\
    \lesssim& \int_0^1  \frac{b}{s^2+b^2}\;ds+\int_1^\infty  e^{-(1-\alpha)s}\;ds\\
    \lesssim&1
  \end{align*}
  where $b=\sin^2\big(\tfrac{\theta_1-\theta_2+\pi}2\big).$
The proof of Proposition \ref{prop:pointest} is now complete.
\end{proof}

\section{Construction of the fundamental solution}

In this section, we construct a fundamental solution of wave equation in an Aharonov-Bohm field based on Cheeger-Taylor's argument \cite{CT1,CT2},
and then we give point-wise estimates for the  fundamental solution.

\subsection{Representation of fundamental solution}

For our purpose of establishing the fundamental solution, we need the following lemma which is given in \cite{Taylor}.
\begin{lemma}
Define the kernel
\begin{equation}\label{qu:kerenelknu}
\begin{split}
 K_{\nu}(t, r_1,r_2)&=\int_0^\infty  \frac{\sin(t\rho)}{\rho}  J_{\nu}(r_1\rho) J_{\nu}(r_2\rho)\rho d\rho
 \end{split}
\end{equation}
where $J_\nu(r)$ is the Bessel function of order $\nu>-1/2$.
Then
\begin{equation}\label{equ:askerenel}
\begin{split}
 K_{\nu}(t, r_1,r_2)&=-\frac1\pi (r_1r_2)^{-\frac12}\lim_{\epsilon\to0} \mathrm{Im}\, Q_{\nu-\frac12}\left(\frac{r_1^2+r_2^2+(\epsilon+it)^2}{2r_1r_2}\right)
\end{split}
\end{equation}
and
where $Q_{\nu-\frac12}(z)$ is a Legendre function which has the representation
\begin{equation}
Q_{\nu-\frac12}(z)=\int_{\cosh^{-1}(z)}^\infty \frac{e^{-s\nu}}{\sqrt{2\cosh(s)-2z}} ds.
\end{equation}
Furthermore, there holds
\begin{align}\label{equ:kernelk}
 & K_{\nu}(t,r_1,r_2)\\\nonumber
 =&\begin{cases}
 0,\quad&\text{if}\quad(t,r_1,r_2)\in I\\
 \frac1{\pi}\int_0^{\beta_1}\frac{\cos(\nu s)}{\sqrt{t^2-r_1^2-r_2^2+2r_1r_2\cos(s)}}\;ds\quad&\text{if}\quad (t,r_1,r_2)\in II\\
 \frac{\cos(\pi\nu)}{\pi}\int_{\beta_2}^\infty\frac{e^{-s\nu}}{\sqrt{r_1^2+r_2^2+2r_1r_2\cosh(s)-t^2}}\;ds\;\quad&\text{if}\quad (t,r_1,r_2)\in III\\
 \end{cases}
\end{align}
where
\begin{align}
I:&=\{(t,r_1,r_2)\in [0,\infty)^3: t<|r_1-r_2|\}, \\
II:&=\{(t,r_1,r_2)\in [0,\infty)^3: |r_1-r_2|<t<r_1+r_2\},\\
III:&=\{(t,r_1,r_2)\in [0,\infty)^3: r_1+r_2<t\},
\end{align}
and
\begin{equation}\label{equ:beta12}
  \beta_1=\cos^{-1}\Big(\frac{r_1^2+r_2^2-t^2}{2r_1r_2}\Big),\quad \beta_2=\cosh^{-1}\Big(\frac{t^2-r_1^2-r_2^2}{2r_1r_2}\Big).
\end{equation}
Moreover, in the region where $t>r_1+r_2$, $K_{\nu}(t,r_1,r_2)$ has another representation
\begin{equation}\label{equ:knuregioniiis}
\begin{split}
  \frac1\pi\bigg\{\int_0^\pi&\frac{\cos(s\nu)}{\sqrt{t^2-r_1^2-r_2^2+2r_1r_2\cos(s)}}\;ds\\\quad&-\sin(\pi\nu)\int_0^{\beta_2}
  \frac{e^{-s\nu}}{\sqrt{t^2-r_1^2-r_2^2-2r_1r_2\cosh(s)}}\;ds\bigg\}.
  \end{split}
\end{equation}

\end{lemma}

Our main result of this section is the following.
\begin{proposition}\label{prop:kernel}  Let $K(t,x,y)$ be the Schwartz kernel of the operator  $\frac{\sin(t\sqrt{\LL_{{\A},0}})}{\sqrt{\LL_{{\A},0}}}$. Suppose $x=r_1(\cos\theta_1,\sin\theta_1)$
and $y=r_2(\cos\theta_2,\sin\theta_2)$ and define
\begin{equation}\label{gamma}
\gamma=\frac{r_1^2+r_2^2-t^2}{2r_1r_2}=\frac{(r_1+r_2)^2-t^2}{2r_1r_2}-1=\frac{(r_1-r_2)^2-t^2}{2r_1r_2}+1
\end{equation}
and
\begin{equation}\label{beta}
\beta_1= \cos^{-1}\big(\frac{r_1^2+r_2^2-t^2}{2r_1r_2}), \quad \beta_2=\cosh^{-1}\big(\frac{t^2-r_1^2-r_2^2}{2r_1r_2}\big).
\end{equation}
Then when $t\geq 0$, the kernel can be written as  a “geometric” term $G(t,r_1,\theta_1,r_2,\theta_2)$ and a “diffractive” term $D(t,r_1,\theta_1,r_2,\theta_2)$
\begin{equation}\label{equ:kerenel}
\begin{split}
K(t,x,y)=K(t,r_1,\theta_1,r_2,\theta_2)=G_w(t,r_1,\theta_1,r_2,\theta_2)+D_w(t,r_1,\theta_1,r_2,\theta_2)
\end{split}
\end{equation}
where
\begin{equation}\label{equ:kerenelG}
\begin{split}
&G_w(t,r_1,\theta_1,r_2,\theta_2)\\
=&\frac1{2\pi}\Big(t^2-(r_1^2+r_2^2-2r_1r_2\cos(\theta_1-\theta_2))\Big)^{-1/2}
e^{i\int_{\theta_2}^{\theta_{1}}\alpha(\theta') d\theta'} \\ &\times \Big\{\mathbbm{1}_{(|r_1-r_2|, r_1+r_2)}(t)
\big[\mathbbm{1}_{[0,\beta_1]}(|\theta_1-\theta_2|)+e^{-2\alpha\pi i}\mathbbm{1}_{[2\pi-\beta_1,2\pi]}(|\theta_1-\theta_2|)\big]\\
&\qquad+\mathbbm{1}_{(r_1+r_2,\infty)}(t)\big[\mathbbm{1}_{[0,\pi]}(|\theta_1-\theta_2|)+e^{-2\alpha\pi i} \mathbbm{1}_{[\pi,2\pi]}(|\theta_1-\theta_2|)\big]\Big\}
\end{split}
\end{equation}
and
\begin{equation}\label{equ:kerenelD}
\begin{split}
&D_w(t,r_1,\theta_1,r_2,\theta_2)=
\frac{\mathbbm{1}_{(r_1+r_2,\infty)}(t)}\pi e^{-i\big(\alpha(\theta_1-\theta_2)-\int_{\theta_2}^{\theta_{1}}\alpha(\theta') d\theta'\big)}\\&\quad\times \int_0^{\beta_2}\Big(t^2-r_1^2-r_2^2-2r_1r_2\cosh s\Big)^{-1/2}
\Big(\sin(|\alpha|\pi)e^{-|\alpha|s}\\&+\sin(\alpha\pi) \frac{(e^{-s}-\cos(\theta_1-\theta_2+\pi))\sinh(\alpha s)+i\sin(\theta_1-\theta_2+\pi)\cosh(\alpha s)}{\cosh(s)
-\cos(\theta_1-\theta_2+\pi)}\Big)ds.
\end{split}
\end{equation}
When $t\leq0$, the similar conclusion hold for \eqref{equ:kerenelG} and \eqref{equ:kerenelD} with replacing $t$ by $-t$.
\end{proposition}

\begin{remark}\label{rem:a0} If $\alpha=0$, the Diffractive term $D$ vanishes and the geometric term $G$ consists with the fundamental solution of wave equation without potential
\begin{equation}
\frac{\sin(t\sqrt{-\Delta})}{\sqrt{-\Delta}}(x,y)=\frac1{2\pi}H(t^2-|x-y|^2)(t^2-|x-y|^2)^{-1/2}
\end{equation}
where $H$ is the Heaviside step function on $\R$. Indeed, in the coordinator $x=r_1(\cos\theta_1,\sin\theta_1)$ and $y=r_2(\cos\theta_2,\sin\theta_2)$, we see
$$|x-y|^2=r_1^2+r_2^2-2r_1r_2\cos(\theta_1-\theta_2).$$ Hence we only need to claim
\begin{equation}\label{equ:claim123}
\begin{split}
&\mathbbm{1}_{\{t^2> r_1^2+r_2^2-2r_1r_2\cos(\theta_1-\theta_2)\}}(t)\\=
&\mathbbm{1}_{(|r_1-r_2|, r_1+r_2)}(t)
\big(\mathbbm{1}_{[0,\beta_1]}(|\theta_1-\theta_2|)+\mathbbm{1}_{[2\pi-\beta_1,2\pi]}(|\theta_1-\theta_2|)\big)\\
&+\mathbbm{1}_{(r_1+r_2,\infty)}(t)\big(\mathbbm{1}_{[0,\pi]}(|\theta_1-\theta_2|)+\mathbbm{1}_{[\pi,2\pi]}(|\theta_1-\theta_2|)\big).
\end{split}
\end{equation}
On the one hand, note $\theta_1,\theta_2\in [0,2\pi]$, then
$$|\theta_1-\theta_2|\leq 2\pi.$$
which implies
$$\mathbbm{1}_{(r_1+r_2,\infty)}(t)\big(\mathbbm{1}_{[0,\pi]}(|\theta_1-\theta_2|)+\mathbbm{1}_{[\pi,2\pi]}(|\theta_1-\theta_2|)=\mathbbm{1}_{(r_1+r_2,\infty)}(t).$$
On the other hand, we observe that
$$t^2> r_1^2+r_2^2-2r_1r_2\cos(\theta_1-\theta_2)\geq (r_1-r_2)^2.$$
Thus it suffices to verify that the two sets
$$\{ t^2> r_1^2+r_2^2-2r_1r_2\cos(\theta_1-\theta_2)\}\iff \{|\theta_1-\theta_2|\in [0,\beta_1]\cup [2\pi-\beta_1,2\pi]\} $$
provided $|r_1-r_2|<t<r_1+r_2$ and $|\theta_1-\theta_2|\leq 2\pi$. We verify this by noticing $\beta_1\in [0,\pi]$ given in \eqref{beta}.
\end{remark}

\begin{proof} We follows the similar argument of Proposition \ref{prop:heatkernel}, but we need to modify the details. 
Let $\nu=\nu_k=|k+\alpha|$ with $k\in\Z$, then the kernel of the operator  $\frac{\sin(t\sqrt{\LL_{{\A},0}})}{\sqrt{\LL_{{\A},0}}}$
\begin{equation}\label{equ:kernsina}
K(t,x,y)=K(t,r_1,\theta_1,r_2,\theta_2)=\sum_{k\in\Z}\varphi_{k}(\theta_1)\overline{\varphi_{k}(\theta_2)}K_{\nu_k}(t,r_1,r_2).
\end{equation}
where $\varphi_{k}$ is the eigenfunction in \eqref{eig-f} and now $K_{\nu}$ is replaced by \eqref{qu:kerenelknu}
\begin{equation}\label{equ:knukdef1}
  K_{\nu}(t,r_1,r_2)=\int_0^\infty \frac{\sin(t\rho)}{\rho} J_{\nu}(r_1\rho)J_{\nu}(r_2\rho) \,\rho d\rho.
\end{equation}
Using \eqref{equ:kernelk}, we write
\begin{equation}
 K_{\nu}(t, r_1,r_2)=K^{I}_{\nu}(t, r_1,r_2)+K^{II}_{\nu}(t, r_1,r_2)+K^{III}_{\nu}(t, r_1,r_2).
 \end{equation}
 where $K^{I}, K^{II}$ and $K^{III}$ are defined in the region $I,II$ and $III$ respectively.
Due to $K^{I}_{\nu}(t, r_1,r_2)=0$, we only consider
 \begin{equation}\label{KIIKIII}
 \begin{split}&\sum_{k\in\Z} \frac1{2\pi}e^{-i\big((\theta_1-\theta_2)(k+\alpha)-\int_{\theta_2}^{\theta_{1}}A(\theta') d\theta'\big)} (K^{II}_{\nu_k}(t, r_1,r_2)+K^{III}_{\nu_k}(t, r_1,r_2)), \\
=&\frac1{2\pi} e^{-i\alpha(\theta_1-\theta_2)+i\int_{\theta_2}^{\theta_{1}}A(\theta') d\theta'} \sum_{k\in\Z} e^{-ik(\theta_1-\theta_2)}  (K^{II}_{\nu_k}+K^{III}_{\nu_k}).
 \end{split}
 \end{equation}
Recall $\nu=|k+\alpha|, k\in\Z$, note that on the line, we have that
\begin{equation}
\begin{split}
\sum_{k\in\Z} \cos(s\nu) e^{-ik(\theta_1-\theta_2)}&=\sum_{k\in\Z} \frac{e^{i(k+\alpha)s}+e^{-i(k+\alpha)s}}2 e^{-ik(\theta_1-\theta_2)}\\
&=\frac12\big(e^{-i\alpha s}\delta(\theta_1-\theta_2+s)+e^{i\alpha s}\delta(\theta_1-\theta_2-s)\big).
\end{split}
\end{equation}
 To get $\cos(s\sqrt{L_{{\A},0}})$ on $\R/(2\pi\Z)$,  by the method of image in \cite{Taylor}, we make the above periodic
 \begin{equation}
 \begin{split}
& \cos(s\sqrt{L_{{\A},0}})\delta(\theta_1-\theta_2)\\=&
 \frac12\sum_{j\in\Z}\big[e^{-is\alpha}\delta(\theta_1+2j\pi-\theta_2+s)+e^{is\alpha}\delta(\theta_1+2j\pi-\theta_2-s)\big].
 \end{split}
 \end{equation}

Therefore, if $|r_1-r_2|<t<r_1+r_2$, by \eqref{equ:kernelk}, we obtain
\begin{align*}
&\sum_{k\in\Z} e^{ik(\theta_1-\theta_2)}K^{II}_{\nu(k)}(t, r_1,r_2)\\\nonumber
= & \frac1{2\pi}\sum_{k\in\Z}e^{-ik(\theta_1-\theta_2)} \int_0^{\beta_1}\frac{\cos(\nu_k s)}{\sqrt{t^2-r_1^2-r_2^2+2r_1r_2\cos(s)}}\;ds\\\nonumber
 =&\int_0^{\beta_1}\frac{\cos(s\sqrt{L_{{\A},0}})\delta(\theta_1-\theta_2)}{\sqrt{t^2-r_1^2-r_2^2+2r_1r_2\cos(s)}}\;ds\\\nonumber
 =&\frac12\sum_{j\in\Z}\int_0^{\beta_1}\frac{1}{\sqrt{t^2-r_1^2-r_2^2+2r_1r_2\cos(s)}}\\
 &\qquad\qquad\times\big[e^{-is\alpha}\delta(\theta_1-\theta_2+2j\pi+s)
 +e^{is\alpha}\delta(\theta_1-\theta_2+2j\pi-s)\big]\;ds\\\nonumber
 =&\sum_{\{j\in\Z: 0\leq |\theta_1-\theta_2+2j\pi|\leq \beta_1\}} \big[t^2-r_1^2-r_2^2+2r_1r_2\cos(\theta_1-\theta_2+2j\pi)\big]^{-\frac12}e^{i(\theta_1-\theta_2+2j\pi)\alpha}
\\
 =&\begin{cases}0\quad&\text{if}\quad \beta_1<|\theta_1-\theta_2|<2\pi-\beta_1\\
 \big[t^2-r_1^2-r_2^2+2r_1r_2\cos(\theta_1-\theta_2)\big]^{-\frac12}e^{i(\theta_1-\theta_2)\alpha}\quad&\text{if}\quad |\theta_1-\theta_2|<\beta_1\\
 \big[t^2-r_1^2-r_2^2+2r_1r_2\cos(\theta_1-\theta_2)\big]^{-\frac12}e^{i(\theta_1-\theta_2-2\pi)\alpha}\quad&\text{if}\quad 2\pi-\beta_1<|\theta_1-\theta_2|<2\pi.
 \end{cases}
\end{align*}
By multiplying $\frac1{2\pi} e^{-i\alpha(\theta_1-\theta_2)+i\int_{\theta_2}^{\theta_{1}}A(\theta') d\theta'}$, it leads to the term in $G$
\begin{equation}\label{equ:kerenelG1}
\begin{split}
&\frac1{2\pi}\Big(t^2-(r_1^2+r_2^2-2r_1r_2\cos(\theta_1-\theta_2))\Big)^{-1/2}
e^{i\int_{\theta_2}^{\theta_{1}}\alpha(\theta') d\theta'} \\ &\times\mathbbm{1}_{(|r_1-r_2|, r_1+r_2)}(t)
\big[\mathbbm{1}_{[0,\beta_1]}(|\theta_1-\theta_2|)+e^{-2\alpha\pi i}\mathbbm{1}_{[2\pi-\beta_1,2\pi]}(|\theta_1-\theta_2|)\big].
\end{split}
\end{equation}
From \eqref{equ:knuregioniiis}, we now consider
 \begin{align}
\nonumber&\sum_{k\in\Z} e^{-ik(\theta_1-\theta_2)}K^{III}_{\nu_k}(t, r_1,r_2)\\ \label{KIII1}
=& \frac{1}{2\pi}\sum_{k\in\Z}e^{-ik(\theta_1-\theta_2)}\int_0^\pi\frac{\cos(s\nu_k)}{\sqrt{t^2-r_1^2-r_2^2+2r_1r_2\cos(s)}}\;ds\\\label{KIII2}
   &+\frac{1}{2\pi}\sum_{k\in\Z}e^{-ik(\theta_1-\theta_2)}\sin(\pi\nu_k)\int_0^{\beta_2}
  \frac{e^{-s\nu_k}}{\sqrt{t^2-r_1^2-r_2^2-2r_1r_2\cosh(s)}}\;ds.
\end{align}
By using the same argument as in $K^{II}_{\nu}$ and multiplying $\frac1{2\pi} e^{-i\alpha(\theta_1-\theta_2)+i\int_{\theta_2}^{\theta_{1}}A(\theta') d\theta'}$ with \eqref{KIII1},
we have the term in $G_w$ such that
\begin{equation}\label{equ:kerenelG2}
\begin{split}
&\frac1{2\pi}\Big(t^2-(r_1^2+r_2^2-2r_1r_2\cos(\theta_1-\theta_2))\Big)^{-1/2}
e^{i\int_{\theta_2}^{\theta_{1}}\alpha(\theta') d\theta'} \\ &\times \mathbbm{1}_{(r_1+r_2,\infty)}(t)\big(\mathbbm{1}_{[0,\pi]}(|\theta_1-\theta_2|)+e^{-2\alpha\pi i} \mathbbm{1}_{[\pi,2\pi]}(|\theta_1-\theta_2|)\big).
\end{split}
\end{equation}
Therefore it remains to consider \eqref{KIII2}.
Recall $\nu_k=|k+\alpha|, k\in\Z$ and by unitary equivalence we can reduce matters to consider the case $\alpha\in(-1,1)\setminus\{0\}$,
then
\begin{equation}
\nu=|k+\alpha|=
\begin{cases}
k+\alpha,\qquad &k\geq1;\\
|\alpha|,\qquad &k=0;\\
-(k+\alpha),\qquad &k\leq -1.
\end{cases}
\end{equation}
Therefore we obtain
\begin{equation}
\begin{split}
\sin(\pi\nu_k)=\sin(\pi|k+\alpha|)=\begin{cases}
\cos(k\pi)\sin(\alpha\pi),\qquad &k\geq1;\\
\sin(|\alpha|\pi),\qquad &k=0;\\
-\cos(k\pi)\sin(\alpha\pi),\qquad &k\leq -1.
\end{cases}
\end{split}
\end{equation}
Therefore we furthermore have
\begin{equation}
\begin{split}
&\sum_{k\in\Z}\sin(\pi|k+\alpha|)e^{-s|k+\alpha|}e^{-ik(\theta_1-\theta_2)}\\
=&\sin(\alpha\pi)\sum_{k\geq 1} \frac{e^{ik\pi}+e^{-ik\pi}}2 e^{-s(k+\alpha)}e^{-ik(\theta_1-\theta_2)}+\sin(|\alpha|\pi)e^{-|\alpha|s}\\&-\sin(\alpha\pi)\sum_{k\leq-1} \frac{e^{ik\pi}+e^{-ik\pi}}2 e^{s(k+\alpha)}e^{-ik(\theta_1-\theta_2)}\\
=&\sin(|\alpha|\pi)e^{-|\alpha|s}+\frac{\sin(\alpha\pi)}2\Big(e^{-s\alpha}\sum_{k\geq1} e^{-ks} \big(e^{-ik(\theta_1-\theta_2+\pi)} +e^{-ik(\theta_1-\theta_2-\pi)}\big)\\&\qquad\qquad\qquad- e^{s\alpha} \sum_{k\geq 1} e^{-ks} \big(e^{ik(\theta_1-\theta_2+\pi)} +e^{ik(\theta_1-\theta_2-\pi)}\big)\Big).
\end{split}
\end{equation}
Note that
\begin{equation}
\sum_{k=1}^\infty e^{ikz}=\frac{e^{iz}}{1-e^{iz}},\qquad \mathrm{Im} z>0,
\end{equation}
we finally obtain
\begin{align}
&\sum_{k\in\Z}\sin(\pi|k+\alpha|)e^{-s|k+\alpha|}e^{-ik(\theta_1-\theta_2)}\\\nonumber
=&\sin(|\alpha|\pi)e^{-|\alpha|s}+\frac{\sin(\alpha\pi)}2\Big(\frac{e^{-(1+\alpha)s-i(\theta_1-\theta_2+\pi)}}
{1-e^{-s-i(\theta_1-\theta_2+\pi)}}+\frac{e^{-(1+\alpha)s-i(\theta_1-\theta_2-\pi)}}{1-e^{-s-i(\theta_1-\theta_2-\pi)}}\\\nonumber
&\qquad\qquad\qquad-\frac{e^{-(1-\alpha)s+i(\theta_1-\theta_2+\pi)}}{1-e^{-s+i(\theta_1-\theta_2+\pi)}}-
\frac{e^{-(1-\alpha)s+i(\theta_1-\theta_2-\pi)}}{1-e^{-s+i(\theta_1-\theta_2-\pi)}}\Big)\\\nonumber
=&\sin(|\alpha|\pi)e^{-|\alpha| s}+\sin(\alpha\pi)\frac{(e^{- s}-\cos(\theta_1-\theta_2+\pi))\sinh(\alpha s)
-i\sin(\theta_1-\theta_2+\pi)\cosh(\alpha s)}{\cosh( s)-\cos(\theta_1-\theta_2+\pi)}
\end{align}
which gives the {\it diffractive} term $D_w$.

\end{proof}

\subsection{Pointwise estimates} In this subsection, we prove the kernel estimates.

\begin{proposition}\label{prop:estimadterm} Let $G_w(t,r_1,\theta_1,r_2,\theta_2)$ be in \eqref{equ:kerenelG} and $D_w(t,r_1,\theta_1,r_2,\theta_2)$ be in \eqref{equ:kerenelD}.
Then, in the polar coordinates $x=r_1(\cos\theta_1,\sin\theta_1), y=r_2(\cos\theta_2,\sin\theta_2)$, the following estimates hold:
\begin{equation}\label{equ:estimg}
|G_w(t,r_1,\theta_1,r_2,\theta_2)|\lesssim \frac{1}{\sqrt{t^2-|x-y|^2}},\qquad t^2>|x-y|^2
\end{equation}
and
\begin{equation}\label{equ:estimd}
|D_w(t,r_1,\theta_1,r_2,\theta_2)|\lesssim \frac{1}{\sqrt{t^2-(r_1+r_2)^2}},\qquad t^2>(r_1+r_2)^2.
\end{equation}

\end{proposition}

\begin{proof} Without loss of generality, we may assume $t\geq0$. We first prove \eqref{equ:estimg}.
From \eqref{equ:kerenelG} and Remark \ref{rem:a0}, $|G_w(t,r_1,\theta_1,r_2,\theta_2)|$ is bounded by $G_w(t,r_1,\theta_1,r_2,\theta_2)$ with $\alpha=0$, which is same as
the fundamental solution of wave equation without potential.
Thus we repeat the argument of Remark \ref{rem:a0} to show
\begin{equation}
\begin{split}
|G_w(t,r_1,\theta_1,r_2,\theta_2)|&\lesssim \frac{\sin(t\sqrt{-\Delta})}{\sqrt{-\Delta}}(x,y)\\
&=\frac1{2\pi}H(t^2-|x-y|^2)(t^2-|x-y|^2)^{-1/2},
\end{split}
\end{equation}
hence we prove \eqref{equ:estimg}.
Next we  prove \eqref{equ:estimd}. To this end, from \eqref{equ:kerenelD}, we need estimate three terms:
\begin{align}\label{kernelD1}
\sin(|\alpha|\pi) \int_0^{\beta_2}\Big(t^2-r_1^2-r_2^2-2r_1r_2\cosh \tau\Big)^{-1/2} e^{-|\alpha|\tau}d\tau,
\end{align}
\begin{align}\label{kernelD2}
\sin(\alpha\pi) \int_0^{\beta_2}\Big(t^2-r_1^2-r_2^2-2r_1r_2\cosh \tau\Big)^{-1/2} \frac{(e^{-\tau}-\cos(\theta_1-\theta_2+\pi))\sinh(\alpha\tau)}{\cosh(\tau)
-\cos(\theta_1-\theta_2+\pi)}d\tau,
\end{align}
and
\begin{align}\label{kernelD3}
\sin(\alpha\pi) \int_0^{\beta_2}\Big(t^2-r_1^2-r_2^2-2r_1r_2\cosh \tau\Big)^{-1/2} \frac{(\sin(\theta_1-\theta_2+\pi)\cosh(\alpha\tau)}{\cosh(\tau)
-\cos(\theta_1-\theta_2+\pi)}d\tau.
\end{align}
{\bf Contribution of \eqref{kernelD1}. }We first consider \eqref{kernelD1}
\begin{align}\label{kerenel}
& \int_0^{\beta_2}\Big(t^2-r_1^2-r_2^2-2r_1r_2\cosh \tau\Big)^{-1/2} e^{-|\alpha|\tau}d\tau\\\nonumber
 =&\frac1{(2r_1r_2)^{1/2}} \int_0^{\beta_2}\Big(\frac{t^2-r_1^2-r_2^2}{2r_1r_2}-\cosh \tau\Big)^{-1/2} e^{-|\alpha|\tau}d\tau\\\nonumber
 =&\frac1{(2r_1r_2)^{1/2}} \int_0^{\beta_2}\Big(\cosh(\beta_2)-\cosh \tau\Big)^{-1/2} e^{-|\alpha|\tau}d\tau\\\nonumber
  =&\frac1{(2r_1r_2)^{1/2}} \int_0^{\beta_2}\frac1{\sqrt{2\sinh\frac{\beta_2+\tau}2\sinh\frac{\beta_2-\tau}2}} e^{-|\alpha|\tau}d\tau \\\nonumber
   =&\frac1{\sinh(\frac{\beta_2}2)(2r_1r_2)^{1/2}} \int_0^{\beta_2}\frac{\sinh(\frac{\beta_2}2)}{e^{|\alpha|\tau}\sqrt{2\sinh\frac{\beta_2+\tau}2\sinh\frac{\beta_2-\tau}2}} d\tau,\\\nonumber
   =&\frac{\sqrt{2}}{\sqrt{t^2-(r_1+r_2)^2}} \int_0^{\beta_2}\frac{\sinh(\frac{\beta_2}2)}{e^{|\alpha|\tau}\sqrt{2\sinh\frac{\beta_2+\tau}2\sinh\frac{\beta_2-\tau}2}} d\tau
\end{align}
where we have used
$$\cosh\alpha-\cosh\beta=2\sinh\frac{\alpha+\beta}2\sinh\frac{\alpha-\beta}2$$
and
\begin{align}\label{defsinhas}
  \sinh\Big(\frac{\beta_2}2\Big)=&\sqrt{\cosh^2\Big(\frac{\beta_2}2\Big)-1}=\sqrt{\frac{\cosh(\beta_2)-1}2}\\\nonumber
  =&\frac1{\sqrt{2}}
\sqrt{\frac{t^2-r_1^2-r_2^2}{2r_1r_2}-1}=\frac1{\sqrt{2}}\sqrt{\frac{t^2-(r_1+r_2)^2}{2r_1r_2}}.
\end{align}

Now we claim that
\begin{equation}
\begin{split}
 \int_0^{\beta_2}\frac{\sinh(\frac{\beta_2}2)}{e^{|\alpha|\tau}\sqrt{2\sinh\frac{\beta_2+\tau}2\sinh\frac{\beta_2-\tau}2}} d\tau \lesssim 1.
\end{split}
\end{equation}

{\bf Case 1: $\beta_2\leq 1$.} In this case, we have
\begin{equation}\label{equ:case1beta}
\begin{split}
& \int_0^{\beta_2}\frac{\sinh(\frac{\beta_2}2)}{e^{|\alpha|\tau}\sqrt{2\sinh\frac{\beta_2+\tau}2\sinh\frac{\beta_2-\tau}2}} d\tau \\
 \lesssim&  \int_0^{\beta_2}\frac{\beta_2}{\sqrt{\beta_2+\tau}\sqrt{\beta_2-\tau}} d\tau \\
 \lesssim&  \int_0^{\beta_2}\frac{\sqrt{\beta_2}}{\sqrt{\beta_2-\tau}} d\tau \lesssim \int_0^{\beta_2}\frac{1}{\sqrt{\beta_2-\tau}} d\tau \lesssim 1.
\end{split}
\end{equation}

{\bf Case 2: $\beta_2\geq 1$.} We obtain
\begin{equation}
\begin{split}
& \int_0^{\beta_2}\frac{\sinh(\frac{\beta_2}2)}{e^{|\alpha|\tau}\sqrt{2\sinh\frac{\beta_2+\tau}2\sinh\frac{\beta_2-\tau}2}} d\tau \\
 \lesssim & \int_0^{\beta_2}\frac{e^{\frac{\beta_2}2}}{e^{|\alpha|\tau}\sqrt{\big(e^{\frac{\beta_2+\tau}2}-e^{-\frac{\beta_2+\tau}2}\big)
 \big(e^{\frac{\beta_2-\tau}2}-e^{-\frac{\beta_2-\tau}2}\big)}}d\tau \\ \lesssim & \int_0^{\beta_2}\frac{e^{\frac{\beta_2}2}}{e^{|\alpha|\tau}\sqrt{e^{\frac{\beta_2+\tau}2}\big(e^{\frac{\beta_2-\tau}2}
 -e^{-\frac{\beta_2-\tau}2}\big)}}d\tau
 \\=&\int_0^{\beta_2}\frac{1}{e^{|\alpha|\tau}\sqrt{1-e^{\tau-\beta_2}}}d\tau
 \\=&\int_0^{\beta_2-\frac12}\frac{1}{e^{|\alpha|\tau}\sqrt{1-e^{\tau-\beta_2}}}d\tau+\int_{\beta_2-\frac12}^{\beta_2}
 \frac{1}{e^{|\alpha|\tau}\sqrt{1-e^{\tau-\beta_2}}}d\tau
  \\\lesssim & \int_0^{\beta_2-\frac12}e^{-|\alpha|\tau}d\tau+\int_{\beta_2-\frac12}^{\beta_2}\frac{1}{e^{|\alpha|\tau}\sqrt{\beta_2-\tau}}d\tau\lesssim 1.
  \end{split}
\end{equation}
This together with \eqref{kerenel} and \eqref{equ:case1beta} yields that
\begin{equation}\label{equ:firkerenel}
\begin{split}
 \int_0^{\beta_2}\Big(t^2-r_1^2-r_2^2-2r_1r_2\cosh \tau\Big)^{-1/2} e^{-|\alpha|\tau}d\tau\lesssim \frac{1}{\sqrt{t^2-(r_1+r_2)^2}}.
\end{split}
\end{equation}

{\bf Contribution of \eqref{kernelD2}. } Note that
$$\cosh(\tau)
-\cos(\theta_1-\theta_2+\pi)=\sinh^2\big(\tfrac{\tau}2\big)+\sin^2\big(\tfrac{\theta_1-\theta_2+\pi}2\big),$$
and by \eqref{defsinhas}, we are going to estimate
\begin{align}\label{equ:cosnique}
&\int_0^{\beta_2}\Big(t^2-r_1^2-r_2^2-2r_1r_2\cosh \tau\Big)^{-1/2}  \frac{\sinh(\alpha\tau)\big(e^{-\tau}-\cos(\theta_1-\theta_2+\pi)\big)}{\sinh^2(\frac{\tau}2)+\sin^2(\frac{\theta_1-\theta_2+\pi}2)}d\tau\\\nonumber
=&\frac{\sqrt{2}}{\sqrt{t^2-(r_1+r_2)^2}} \int_0^{\beta_2}\frac{\sinh(\frac{\beta_2}2)}{\sqrt{2\sinh\frac{\beta_2+\tau}2\sinh\frac{\beta_2-\tau}2}}  \frac{\sinh(\alpha\tau)\big(e^{-\tau}-\cos(\theta_1-\theta_2+\pi)\big)}{\sinh^2(\frac{\tau}2)+\sin^2(\frac{\theta_1-\theta_2+\pi}2)}d\tau\\\nonumber
=&\frac{1}{\sqrt{t^2-(r_1+r_2)^2}} \int_0^{\beta_2}\frac{\sinh(\frac{\beta_2}2)}{\sqrt{2\sinh\frac{\beta_2+\tau}2\sinh\frac{\beta_2-\tau}2}}  \frac{\sinh(\alpha\tau)\big(e^{-\tau}-1+1-\cos(\theta_1-\theta_2+\pi)\big)}{\sinh^2(\frac{\tau}2)+\sin^2(\frac{\theta_1-\theta_2+\pi}2)}
d\tau\\\nonumber
=& \frac{1}{\sqrt{t^2-(r_1+r_2)^2}} \int_0^{\beta_2}\frac{\sinh(\frac{\beta_2}2)}{\sqrt{2\sinh\frac{\beta_2+\tau}2\sinh\frac{\beta_2-\tau}2}}  \frac{\sinh(\alpha\tau)\big(e^{-\tau}-1-2\sin^2(\frac{\theta_1-\theta_2+\pi}2)\big)}{\sinh^2(\frac{\tau}2)
+\sin^2(\frac{\theta_1-\theta_2+\pi}2)}d\tau.
\end{align}
To estimate it, we divide into two cases.

{\bf Case 1:} $\beta_2\leq 1$. On one hand, we have
\begin{align}\label{case1term21}
& \int_0^{\beta_2}\frac{\sinh(\frac{\beta_2}2)}{\sqrt{2\sinh\frac{\beta_2+\tau}2\sinh\frac{\beta_2-\tau}2}}  \frac{\sinh(\alpha\tau)\big(e^{-\tau}-1\big)}{\sinh^2(\frac{\tau}2)+\sin^2(\frac{\theta_1-\theta_2+\pi}2)}d\tau\\\nonumber
\lesssim& \int_0^{\beta_2}
\frac{\sqrt{\beta_2}}{\sqrt{\beta_2-\tau}}  \frac{(\alpha\tau)\tau}{\tau^2+(\frac{\theta_1-\theta_2+\pi}2)^2}d\tau\\\nonumber
\lesssim&  \int_0^{\beta_2}
\frac{\sqrt{\beta_2}}{\sqrt{\beta_2-\tau}} d\tau\lesssim 1
\end{align}
Similarly, we get
\begin{align}\label{case1term22}
& \int_0^{\beta_2}\frac{\sinh(\frac{\beta_2}2)}{\sqrt{2\sinh\frac{\beta_2+\tau}2\sinh\frac{\beta_2-\tau}2}}  \frac{\sinh(\alpha\tau)\sin^2(\frac{\theta_1-\theta_2+\pi}2)}{\sinh^2(\frac{\tau}2)+\sin^2(\frac{\theta_1-\theta_2+\pi}2)}d\tau\\\nonumber
\lesssim&  \int_0^{\beta_2}
\frac{\sqrt{\beta_2}}{\sqrt{\beta_2-\tau}}  \frac{(\frac{\theta_1-\theta_2+\pi}2)^2}{\tau^2+(\frac{\theta_1-\theta_2+\pi}2)^2}d\tau\\\nonumber
\lesssim&  \int_0^{\beta_2}
\frac{\sqrt{\beta_2}}{\sqrt{\beta_2-\tau}} d\tau\lesssim 1
\end{align}
This together with \eqref{equ:cosnique}--\eqref{case1term22} implies
$$\eqref{kernelD2}\lesssim \frac{1}{\sqrt{t^2-(r_1+r_2)^2}} $$
 in the case that $\beta_2\leq1.$

{\bf Case 2:} $\beta_2\geq 1$. By the above argument, we have the estimate for  $\int_0^1$. So we only need to consider $\int_1^{\beta_2}$.
We have
\begin{align}\label{equ:case2term31}
& \int_1^{\beta_2}\frac{\sinh(\frac{\beta_2}2)}{\sqrt{2\sinh\frac{\beta_2+\tau}2\sinh\frac{\beta_2-\tau}2}}  \frac{\sinh(\alpha\tau)|(e^{-\tau}-1-2\sin^2(\frac{\theta_1-\theta_2+\pi}2))|}
{\sinh^2(\frac{\tau}2)+\sin^2(\frac{\theta_1-\theta_2+\pi}2)}d\tau\\\nonumber
\lesssim&  \int_1^{\beta_2}
\frac{e^{\frac{\beta_2}2}}{\sqrt{e^{\frac{\beta_2+\tau}2}(e^{\frac{\beta_2-\tau}2}-e^{\frac{\tau-\beta_2}2})}}   \frac{\sinh(\alpha\tau)}{\sinh^2(\frac{\tau}2)+\sin^2(\frac{\theta_1-\theta_2+\pi}2)}d\tau\\\nonumber
\lesssim&  \int_1^{\beta_2}
\frac{1}{\sqrt{1-e^{\tau-\beta_2}}}   \frac{\sinh(\alpha\tau)}{\sinh^2(\frac{\tau}2)+\sin^2(\frac{\theta_1-\theta_2+\pi}2)}d\tau\\\nonumber
\lesssim&  \Big( \int_1^{\beta_2-1/2}
  \frac{\sinh(\alpha\tau)}{\sinh^2(\frac{\tau}2)+\sin^2(\frac{\theta_1-\theta_2+\pi}2)}d\tau
+\int_{\beta_2-1/2} ^{\beta_2}
\frac{1}{\sqrt{\beta_2-\tau}}   \frac{\sinh(\alpha\tau)}{\sinh^2(\frac{\tau}2)+\sin^2(\frac{\theta_1-\theta_2+\pi}2)}d\tau\Big)\\\nonumber
\lesssim& \Big( \int_1^{\beta_2-1/2}  e^{-(1-\alpha)\tau}d\tau+\int_{\beta_2-1/2} ^{\beta_2}
\frac{1}{\sqrt{\beta_2-\tau}}   e^{-(1-\alpha)\tau}d\tau\Big)\\\nonumber
\lesssim& 1,
\end{align}
and
\begin{align}\label{equ:case2term323}
&\Big| \int_1^{\beta_2}\frac{\sinh(\frac{\beta_2}2)}{\sqrt{2\sinh\frac{\beta_2+\tau}2\sinh\frac{\beta_2-\tau}2}}\frac{\cosh(\alpha\tau)
\sin(\theta_1-\theta_2+\pi)}{\sinh^2(\frac{\tau}2)+\sin^2(\frac{\theta_1-\theta_2+\pi}2)}d\tau\Big|\\\nonumber
&\lesssim \int_1^{\beta_2}
\frac{e^{\frac{\beta_2}2}}{\sqrt{e^{\frac{\beta_2+\tau}2}(e^{\frac{\beta_2-\tau}2}-e^{\frac{\tau-\beta_2}2})}}   \frac{\cosh(\alpha\tau)}{\sinh^2(\frac{\tau}2)}d\tau\\\nonumber
&\lesssim   \int_1^{\beta_2}
\frac{1}{\sqrt{1-e^{\tau-\beta_2}}}   e^{-(1-\alpha)\tau}d\tau\\\nonumber
&\lesssim  \Big( \int_1^{\beta_2-1/2} e^{-(1-\alpha)\tau}d\tau+ \int_{\beta_2-1/2}^{\beta_2} \frac{1}{\sqrt{\beta_2-\tau}}  e^{-(1-\alpha)\tau}d\tau\Big)\\\nonumber
&\lesssim 1.
\end{align}
These together with \eqref{equ:cosnique} yield
$$\eqref{kernelD2}\lesssim \frac{1}{\sqrt{t^2-(r_1+r_2)^2}} $$
 in the case that $\beta_2\geq1.$

{\bf Contribution of \eqref{kernelD3}. }
Similarly as \eqref{kernelD2}, we write
\begin{align}\label{equ:term2sesca}
   &\int_0^{\beta_2}\Big(t^2-r_1^2-r_2^2-2r_1r_2\cosh \tau\Big)^{-1/2} \frac{\cosh(\alpha\tau)\sin(\theta_1-\theta_2+\pi)}{\sinh^2(\frac{\tau}2)
   +\sin^2(\frac{\theta_1-\theta_2+\pi}2)}d\tau\\\nonumber
   =&\frac{1}{\sqrt{t^2-(r_1+r_2)^2}} \int_0^{\beta_2}\frac{\sinh(\frac{\beta_2}2)}{\sqrt{2\sinh\frac{\beta_2+\tau}2\sinh\frac{\beta_2-\tau}2}}
   \frac{\cosh(\alpha\tau)\sin(\theta_1-\theta_2+\pi)}{\sinh^2(\frac{\tau}2)
   +\sin^2(\frac{\theta_1-\theta_2+\pi}2)}d\tau.
\end{align}
We divide into two cases.

{\bf Case 1:} $\beta_2\leq 1$. Let $b:=\big|\sin\big(\tfrac{\theta_1-\theta_2+\pi}2\big)\big|,$ we obtain
\begin{align}\label{equ:cosonlyaf}
& \Big|\int_0^{\beta_2}\frac{\sinh(\frac{\beta_2}2)}{\sqrt{2\sinh\frac{\beta_2+\tau}2\sinh\frac{\beta_2-\tau}2}}   \frac{\cosh(\alpha\tau)\sin(\theta_1-\theta_2+\pi)}{\sinh^2(\frac{\tau}2)+\sin^2(\frac{\theta_1-\theta_2+\pi}2)}d\tau\Big|\\\nonumber
\lesssim&  \int_0^{\beta_2}
\frac{\sqrt{\beta_2}}{\sqrt{\beta_2-\tau}}  \frac{b}{\tau^2+b^2}d\tau\\\nonumber
\lesssim&  \Big( \int_0^{\beta_2/2}  \frac{b}{\tau^2+b^2}d\tau
+\int_{\beta_2/2}^{\beta_2}\frac{\sqrt{\beta_2}}{\sqrt{\beta_2-\tau}}  \frac{b}{\tau^2+b^2}d\tau
\Big)\\\nonumber
\lesssim& \Big( \int_0^{+\infty}  \frac{1}{\tau^2+1}d\tau
+\frac1{\beta_2}\int_{\beta_2/2}^{\beta_2}\frac{\sqrt{\beta_2}}{\sqrt{\beta_2-\tau}}  d\tau
\Big)\\\nonumber
\lesssim& 1.
\end{align}
This together with \eqref{equ:cosnique}--\eqref{case1term22} implies
$$\eqref{kernelD3}\lesssim \frac{1}{\sqrt{t^2-(r_1+r_2)^2}} $$
 in the case that $\beta_2\leq1.$

{\bf Case 2:} $\beta_2\geq 1$. By the above argument, we have the estimate for  $\int_0^1$. So we only need to consider $\int_1^{\beta_2}$.
We have
\begin{align}\label{equ:case2term32}
&\Big| \int_1^{\beta_2}\frac{\sinh(\frac{\beta_2}2)}{\sqrt{2\sinh\frac{\beta_2+\tau}2\sinh\frac{\beta_2-\tau}2}}\frac{\cosh(\alpha\tau)
\sin(\theta_1-\theta_2+\pi)}{\sinh^2(\frac{\tau}2)+\sin^2(\frac{\theta_1-\theta_2+\pi}2)}d\tau\Big|\\\nonumber
&\lesssim \int_1^{\beta_2}
\frac{e^{\frac{\beta_2}2}}{\sqrt{e^{\frac{\beta_2+\tau}2}(e^{\frac{\beta_2-\tau}2}-e^{\frac{\tau-\beta_2}2})}}   \frac{\cosh(\alpha\tau)}{\sinh^2(\frac{\tau}2)}d\tau\\\nonumber
&\lesssim   \int_1^{\beta_2}
\frac{1}{\sqrt{1-e^{\tau-\beta_2}}}   e^{-(1-\alpha)\tau}d\tau\\\nonumber
&\lesssim  \Big( \int_1^{\beta_2-1/2} e^{-(1-\alpha)\tau}d\tau+ \int_{\beta_2-1/2}^{\beta_2} \frac{1}{\sqrt{\beta_2-\tau}}  e^{-(1-\alpha)\tau}d\tau\Big)\\\nonumber
&\lesssim 1.
\end{align}
These together with \eqref{equ:cosnique} and \eqref{equ:term2sesca} yield $$\eqref{kernelD3}\lesssim \frac{1}{\sqrt{t^2-(r_1+r_2)^2}} $$ in the case that $\beta_2\geq1.$

Therefore, we conclude the proof of Proposition \ref{prop:estimadterm}.

\end{proof}

\section{Proof of Theorem \ref{thm:dispersive}}
The main ingredient is the following localized-frequency decay result.
If done, by the definition of Besov space, we sum in $j$ to complete the proof of Theorem \ref{thm:dispersive}.
\begin{proposition}\label{prop:mdispersive}
Let $\varphi\in C_c^\infty(\mathbb{R}\setminus\{0\})$, with $0\leq\varphi\leq1$, and $\text{supp}\,\varphi\subset[1/2,2]$, as in \eqref{dp}. Then for all $j\in \Z$, there exists a constant $C$ independent of $x, y$ and $t$ such that
\begin{equation}\label{mic-dispersive}
\begin{split}
\Big\|\frac{\sin(t\sqrt{\mathcal L_{{\A},0}})}{\sqrt{\mathcal L_{{\A},0}}}f\Big\|_{L^\infty(\R^2)}\leq C 2^{\frac{j}2}(2^{-j}+|t|)^{-\frac12}\|f\|_{L^1(\R^2)},
\end{split}
\end{equation}
where $f=\varphi(2^{-j}{\sqrt{\mathcal L_{{\A},0}}})f$.
\end{proposition}

\begin{proof} By the scaling, it suffices to establish \eqref{mic-dispersive} when $j=0$,  that is,
\begin{equation}\label{mic-dispersive'}
\begin{split}
\Big\|\frac{\sin(t\sqrt{\mathcal L_{{\A},0}})}{\sqrt{\mathcal L_{{\A},0}}}f\Big\|_{L^\infty(\R^2)}\leq C (1+|t|)^{-\frac12}\|f\|_{L^1},\quad f=\varphi({\sqrt{\mathcal L_{{\A},0}}})f.
\end{split}
\end{equation}
By using Proposition \ref{prop:kernel} and Proposition \ref{prop:estimadterm},
we need to establish
\begin{equation}\label{mic-dispersive1}
\begin{split}
\Big\|\int_{|x-y|^2\leq t^2} \frac{|f(y)|}{\sqrt{t^2-|x-y|^2}} dy \Big\|_{L^\infty(\R^2)}\leq C (1+|t|)^{-\frac12}\|f\|_{L^1},
\end{split}
\end{equation}
and
\begin{equation}\label{mic-dispersive2}
\begin{split}
\Big\|\int_{(r_1+r_2)^2\leq t^2}\int_{0}^{2\pi} \frac{|f(r_2,\theta_2)|}{\sqrt{t^2-(r_1+r_2)^2}} r_2dr_2 d\theta_2 \Big\|_{L^\infty(\R^2)}\leq C (1+|t|)^{-\frac12}\|f\|_{L^1}.
\end{split}
\end{equation}
We use the method in Shatah-Struwe \cite[Page 47]{SS} to prove both of them. By the symmetry of $t$, we only consider the case $t\geq0$.

 Now we first prove \eqref{mic-dispersive1}. We write
 \begin{equation}
\begin{split}
&\int_{|x-y|^2\leq t^2} \frac{|f(y)|}{\sqrt{t^2-|x-y|^2}} dy=\int_{|y|\leq t} \frac{|f(x-y)|}{\sqrt{t^2-|y|^2}} dy.
\end{split}
\end{equation}
When $0\leq t\leq 1$, from the Bernstein inequality \eqref{bern1}, it is easy to see
\begin{equation}\label{t<1}
\begin{split}
\int_{|y|\leq t} \frac{|f(x-y)|}{\sqrt{t^2-|y|^2}} dy\leq Ct \|f\|_{L^\infty}\leq C\|f\|_{L^1}.
\end{split}
\end{equation}
Thus it suffices to  prove, for $t\geq1$
 \begin{equation}\label{t>1}
\begin{split}
\Big\|\int_{|y|\leq t} \frac{|f(x-y)|}{\sqrt{t^2-|y|^2}} dy\Big\|_{L^\infty_x}\leq C |t|^{-1/2}\|f\|_{L^1}.
\end{split}
\end{equation}
To this end, we split into two terms
 \begin{equation}
\begin{split}
\int_{|y|\leq t} \frac{|f(x-y)|}{\sqrt{t^2-|y|^2}} dy =\int_{ |y|\leq t-\frac12} \frac{|f(x-y)|}{\sqrt{t^2-|y|^2}} dy+\int_{t-\frac12\leq |y|\leq t} \frac{|f(x-y)|}{\sqrt{t^2-|y|^2}} dy.
\end{split}
\end{equation}
For the first term, we can estimate
 \begin{equation}
\begin{split}
\int_{ |y|\leq \frac12} \frac{|f(x-y)|}{\sqrt{t^2-|y|^2}} dy\leq Ct^{-1/2}\|f\|_{L^\infty}\leq Ct^{-1/2}\|f\|_{L^1}.
\end{split}
\end{equation}
For the second term, we have
 \begin{equation}\label{t>1'}
\begin{split}
&\int_{t-\frac{1}2\leq |y|\leq t} \frac{|f(x-y)|}{\sqrt{t^2-|y|^2}} dy
\leq C t^{-1/2} \|f\|_{L^\infty} \int_{t-\frac{1}2\leq |y|\leq t} \frac{1}{\sqrt{t-|y|}}dy,\\
&= C t^{-1/2} \int_{t-\frac{1}2\leq |y|\leq t} |f(x-y)|\varphi(|y|) \frac{-y}{|y|}\cdot\nabla(\sqrt{t-|y|})dy\\
&\leq C t^{-1/2} \int_{0\leq |y|\leq t} |f(x-y)|\varphi(|y|) \frac{-y}{|y|}\cdot\nabla(\sqrt{t-|y|})dy\\
\end{split}
\end{equation}
where $0\leq \varphi(r)\in C^{\infty}([0,\infty))$ takes value $1$ when $r\in [t-1/2,t]$ and vanishes if $r\in[0,t-2/3]$.
By using integration by parts, we obtain
 \begin{equation}\label{t>1''}
\begin{split}
&\int_{t-\frac{1}2\leq |y|\leq t} \frac{|f(x-y)|}{\sqrt{t^2-|y|^2}} dy\\
&\leq C t^{-1/2}\Big(\int_{0\leq |y|\leq t} \big(|\nabla |f(x-y)||+\frac{|f(x-y)|}{|y|}\big)\varphi(|y|)\sqrt{t-|y|}dy\\
&\qquad+\int_{0\leq |y|\leq t}|f(x-y)|\frac{\varphi'(|y|)}{|y|}\big)\sqrt{t-|y|}dy\Big), \\
&\leq C t^{-1/2}\Big(\int_{\R^2} |\nabla |f(x-y)||dy+\int_{\R^2}|f(x-y)|dy\Big).\\
\end{split}
\end{equation}
By using the diamagnetic inequality (see \cite[Lemma A.1]{FFT})
 $$
|\nabla|f|(x)| \leq |\nabla_{\A} f(x)|,
 $$
 we therefore show
 \begin{equation}
\begin{split}
&\int_{\frac{1}2\leq |y|\leq t} \frac{|f(x-y)|}{\sqrt{t^2-|y|^2}} dy
\leq C t^{-1/2}\big(\|\nabla_{\A}f\|_{L^1}+\|f\|_{L^1}\big),
\end{split}
\end{equation}
which implies \eqref{t>1} due to $f=\varphi({\sqrt{\mathcal L_{{\A},0}}})f$ localized in frequency and \eqref{bern2}. Hence we prove \eqref{mic-dispersive1}.\vspace{0.2cm}

We next prove \eqref{mic-dispersive2}. The argument is similar but we work in polar coordinates. Since we assume $t\geq0$, we consider
\begin{equation}
\begin{split}
\Big|\int_{r_1+r_2\leq t}\int_{0}^{2\pi} \frac{|f(r_2,\theta_2)|}{\sqrt{t^2-(r_1+r_2)^2}} r_2dr_2 d\theta_2 \Big|.
\end{split}
\end{equation}
For $t\leq 1$, similar as \eqref{t<1}, it follows
\begin{equation}
\begin{split}
\Big|\int_{r_1+r_2\leq t}\int_{0}^{2\pi} \frac{|f(r_2,\theta_2)|}{\sqrt{t^2-(r_1+r_2)^2}} r_2dr_2 d\theta_2 \Big|\leq C\|f\|_{L^1}.
\end{split}
\end{equation}
When $t\geq 1$, we have
 \begin{align*}
    &\Big|\int_{r_1+r_2<t}\int_0^{2\pi}\frac{|f(r_2,\theta_2)|}{\sqrt{t^2-(r_1+r_2)^2}}\;r_2\;dr_2\;d\theta\Big|\\
     \lesssim&t^{-\frac12}\int_{r_2<t-r_1}\int_0^{2\pi}\frac{|f(r_2,\theta_2)|}{\sqrt{(t-r_1)-r_2}}\;r_2\;dr_2\;d\theta\Big|\\
     \lesssim&t^{-\frac12}\int_{|y|<t-r_1}\frac{|f(y)|}{\sqrt{(t-r_1)-|y|}}\;dy.
  \end{align*}
Then we repeat the argument of \eqref{t>1'} and \eqref{t>1''} to show \eqref{mic-dispersive2}.

\end{proof}

\section{Proof of Theorem \ref{thm:Stri}}

We devote this section to the proof of the Strichartz estimates, Theorem \ref{thm:Stri}.
First notice that, by the representation formula \eqref{eq:solution} and the equivalence in Lemma \ref{Lem:sobequil2}, it is sufficient to prove the following estimate
\begin{equation}\label{equ:claim}
\|e^{it\sqrt{\LL_{{\A},a}}}f\|_{L^q_tL^{r}_x(\mathbb{R}\times \R^2)}\lesssim
\|f\|_{\dot{H}^s_{{\A},0}(\R^2)}.
\end{equation}

We will first prove \eqref{equ:claim} in the purely magnetic case $a\equiv0$, and then in the general case, as a consequence of a local smoothing estimate.

\subsection{Strichartz estimates for purely magnetic waves}

Let us start with the purely magnetic case $a\equiv0$. Our first step is to prove the following claim
\begin{equation}\label{claim123}
\|e^{it\sqrt{\LL_ {{\A},0}}}f\|_{L^q_tL^{r}_x(\mathbb{R}\times \R^2)}\lesssim
\|f\|_{\dot{H}^s_{{\A},0}(\R^2)},
\end{equation}
for $s\in\R$, any wave-admissible pair $(q,r)\in\Lambda_s^W$ as in \eqref{adm-p-w}, $f\in \dot{H}^s_{{\A},0}(\R^2)$, and for some $C>0$ independent on $f$.

\begin{proposition}\label{lStrichartz} Let $U(t)=e^{it\sqrt{\LL_{{\A},0}}}$ and
$f=\varphi_j(\sqrt{\LL_{{\A},0}})f$ as in \eqref{equ:lidef} for $j\in\Z$,  then
\begin{equation}\label{lstri}
\|U(t)f\|_{L^q_tL^{r}_x(\mathbb{R}\times \R^2)}\lesssim
2^{js}\|f\|_{L^2(\R^2)},
\end{equation}
where $s\in\R$ and $(q,r)\in\Lambda_s^W$ defined in \eqref{adm-p-w}.
\end{proposition}
\begin{proof}
By the scaling, it suffices to prove \eqref{lstri} when $j=0$. Since the dispersive estimate \eqref{mic-dispersive} only works for $\sin(t\sqrt{\LL_{{\A},0}})$ while not $U(t)$,
we define the wave group

\begin{equation}       
{\bf K}(t)=\left(                 
  \begin{array}{cc}   
    \dot K(t) & K(t) \\  
    \ddot K(t) & \dot K(t) \\  
  \end{array}
\right)                 
\end{equation}
where $K(t)=\sin(t\sqrt{\LL_{{\A},0}})/\sqrt{\LL_{{\A},0}}$ and $\dot K(t)=\cos(t\sqrt{\LL_{{\A},0}})$.
Then the wave group ${\bf K}(t)$ acts on pairs $(h,g)$
\begin{equation*}
   {\bf K}(t)  \left(\begin{array}{c}
    h \\
    g \\
  \end{array}
\right)  =\left(
  \begin{array}{cc}
    \dot K(t) & K(t) \\
     \ddot K(t) & \dot K(t) \\
  \end{array}
\right)     \left(\begin{array}{c}
    h \\
    g \\
  \end{array}
\right)    .
\end{equation*}
We claim that
\begin{equation}  \label{claim}
\left\|{\bf K}(t)  \left(\begin{array}{c}
    h \\
    g \\
  \end{array}
\right) \right\|_{L^q_tL^{r}_x(\mathbb{R}\times \R^2)}\\
\lesssim \left\| \left(\begin{array}{c}
    h \\
    g \\
  \end{array}
\right) \right\|_{L^2_x(\R^2)}
\end{equation}
where $h=\varphi_0(\sqrt{\LL_{{\A},0}})h$ and $g=\varphi_0(\sqrt{\LL_{{\A},0}})g$. If we could prove \eqref{claim},
by taking $h=f$ and $g=i\sqrt{\LL_{{\A},0}}f$, then we obtain \eqref{lstri} with $j=0$.\vspace{0.2cm}

We now prove \eqref{claim}. By using the $TT^*$-argument, it suffices to show
\begin{equation}  \label{d-claim}
\left\|\int_{\R}{\bf K}(t-s)  \left(\begin{array}{c}
    H(s) \\
    G(s)\\
  \end{array}
\right)d s \right\|_{L^q_tL^{r}_x(\mathbb{R}\times \R^2)}\\
\lesssim \left\| \left(\begin{array}{c}
    H \\
    G \\
  \end{array}
\right) \right\|_{L^{q'}_tL^{r'}_x(\R\times\R^2)}
\end{equation}
where $H(s)$ and $G(s)$ are functions in $L^{q'}(\R^2)$, for each $s$, $H(s)$ and $G(s)$ are localized to unit frequency.
To this end, since the others are similarly, we only estimate
\begin{equation}  \label{d-claim1}
\left\|\int_{\R}  K(t-s) G(s)d s \right\|_{L^q_tL^{r}_x(\mathbb{R}\times \R^2)}\lesssim \|G\|_{L^{q'}_tL^{r'}_x(\R\times\R^2)},
\end{equation}
and
\begin{equation}  \label{d-claim2}
\left\|\int_{\R} \dot K(t-s) H(s)d s \right\|_{L^q_tL^{r}_x(\mathbb{R}\times \R^2)}\lesssim \|H\|_{L^{q'}_tL^{r'}_x(\R\times\R^2)} .
\end{equation}
By interpolating \eqref{mic-dispersive} ($j=0$) with $L^2$-estimate and the Hardy-Littlewood-Sobolev inequality, we get \eqref{d-claim1}.
For the proof of \eqref{d-claim2}, we refer to \cite[Page 14-15]{BFM}, since the argument is completely analogous, and omit further details.

\end{proof}

\subsection{Local smoothing for wave associated with $\LL_{{\A},a}$}

In view to apply a perturbation argument for the proof of Theorem \ref{thm:Stri}, we need to prove some suitable local smoothing estimates.
\begin{proposition}\label{prop:loc}
Let $a\in W^{1,\infty}(\mathbb{S}^{1},\mathbb{R})$, ${\A}\in
W^{1,\infty}(\mathbb{S}^{1},\mathbb{R}^2)$, and assume \eqref{eq:transversal}, \eqref{equ:condassa}.
Let $L_{{\A},a}$ be the spherical operator in \eqref{L-angle}, with first eigenvalue $\mu_1({\A},a)$ as in \eqref{eig-Aa}, and denote by $\nu_0:=\sqrt{\mu_1({\A},a)}$. Then there exists a constant $C>0$ such that, for any $f\in\dot H^{\beta-\frac12}_{{\A},a}$,
\begin{equation}\label{local-s}
\begin{split}
\|r^{-\beta}e^{it\sqrt{\LL_{{\A},a}}}f\|_{L^2_t(\R;L^2(\R^2))}\leq
C \|f\|_{\dot H^{\beta-\frac12}_{{\A},a}},
\end{split}
\end{equation}
for any $\beta\in\left(\frac12,1+\nu_0\right)$.
\end{proposition}

\begin{remark}\label{rem:loc-s}
The first endpoint $\beta=\frac12$ in \eqref{local-s} is known to be false, even in the free case ${\A}\equiv a\equiv0$, by the usual Agmon-H\"ormander Theory (see e.g. \cite{IS} and the references therein). As for the second endpoint $\beta=1+\nu_0$, this equals 1, in the free case. In the perturbed case, thanks to assumption \eqref{equ:condassa}, we have $\mu_1({\A},a)>0$, hence $\nu_0>0$ is well defined and we get an improvement in the range of validity of the estimate. This fact has been already observed in several papers, for different evolution models (see e.g. \cite{CF,CF1, FFFP2, FGK, GK, MZZ2}). In addition, a further improvement occurs for higher frequencies. Indeed, if $f(x)$, belongs to $\bigoplus_{\nu>k} \mathcal{H}^{\nu}\cap\dot H^{\beta-\frac12}_{{\A},a}(\R^2)$ where $k>\nu_0$, then one
can relax the upper restriction on $\beta$ to  $\beta<1+k$.

\end{remark}

\begin{proof}

Suppose that
\begin{equation}
f(x)
=\sum_{k=0}^\infty a_{k}(r)\psi_{k}(\theta), \qquad b_k(\rho)=(\mathcal{H}_{\nu}a_{k})(\rho).
\end{equation}
We want to estimate
\begin{equation}
\begin{split}
e^{it\sqrt{\LL_{{\A},a}}}f&=\sum_{k=0}^\infty \psi_{k}(\theta)\int_0^\infty J_{\nu_k}(r\rho)e^{
it\rho}b_{k}(\rho)\, \rho d\rho,\quad \nu_k=\sqrt{\mu_k}.
\end{split}
\end{equation}
By the Plancherel theorem with respect to time $t$, it suffices to estimate the term
\begin{equation*}
\begin{split}
&\int_{\R^2}\int_0^\infty\big|\sum_{k=0}^\infty \psi_{k}(\theta)J_{\nu_k}(r\rho)
b_{k}(\rho)\rho \big|^2d\rho |x|^{-2\beta}dx\\
=&\sum_{k=0}^\infty\int_0^\infty\int_0^\infty\big|J_{\nu_k}(r\rho)b_{k}(\rho)
\rho \big|^2 d\rho\, r^{1-2\beta}dr.
\end{split}
\end{equation*}
Let $\chi$ be a smoothing function supported in $[1,2]$, we make dyadic decompositions to obtain
\begin{equation}\label{scal-reduce}
\begin{split}
&\sum_{k=0}^\infty\int_0^\infty\int_0^\infty\big|J_{\nu_k}(r\rho)b_{k}(\rho)
\rho \big|^2 d\rho\, r^{1-2\beta}dr\\
\lesssim&
\sum_{k=0}^\infty\sum_{M\in2^{\Z}}\sum_{R\in2^{\Z}}M^{1+2\beta}R^{1-2\beta}\int_{R}^{2R}
\int_{0}^\infty\big|J_{\nu_k}(r\rho)b_{k}(M\rho)\chi(\rho)
\big|^2 d\rho\, dr.
\end{split}
\end{equation}
Let
\begin{equation}\label{def:Q}
\begin{split}
Q_{k}(R,M)=\int_{R}^{2R}\int_{0}^\infty\big|J_{\nu_k}(r\rho)b_{k}(M\rho)\chi(\rho)
\big|^2 d\rho\,  dr.\end{split}
\end{equation}
We now claim that the following inequality holds:
\begin{equation}\label{est:Q}
Q_{k}(R,M) \lesssim
\begin{cases}
R^{2\nu_k+1}M^{-2}\|b_{k}(\rho)\chi(\frac{\rho}M)\rho^{1/2}\|^2_{L^2},~
R\lesssim 1;\\
M^{-2}\|b_{k}(\rho)\chi(\frac{\rho}M)\rho^{1/2}\|^2_{L^2},~
R\gg1.
\end{cases}
\end{equation}
\begin{proof}[Proof of \eqref{est:Q}]
We consider two different cases.

{\bf $\bullet$ Case 1: $R\lesssim1$.} Since $\rho\sim1$, thus
$r\rho\lesssim1$. By
\eqref{bessel-r}, we obtain
\begin{equation*}
\begin{split}
Q_{k}(R,M)&\lesssim\int_{R}^{2R}\int_{0}^\infty\Big| \frac{
(r\rho)^{\nu}}{2^{\nu}\Gamma(\nu+\frac12)\Gamma(\frac12)}b_{k}(M\rho)\chi(\rho)\Big|^2 d\rho dr\\& \lesssim
R^{2\nu+1}M^{-2}\|b_{k}(\rho)\chi(\frac{\rho}M)\rho^{1/2}\|^2_{L^2}.
\end{split}
\end{equation*}

{\bf $\bullet$ Case 2: $R\gg1$.} Since $\rho\sim1$, thus $r\rho\gg 1$. We
estimate by \eqref{est:b} in Lemma \ref{lem: J}
\begin{equation*}
\begin{split}
Q_{k}(R,M)&\lesssim
\int_{0}^\infty\big|b_{k}(M\rho)\chi(\rho)\big|^2\int_{R}^{2R}\big|J_{\nu}(
r\rho)\big|^2  dr d\rho\\& \lesssim \int_{0}^\infty\big|b_{k}(M\rho)\chi(\rho)\big|^2 d\rho\lesssim
M^{-2}\|b_{k}(\rho)\chi(\frac{\rho}M)\rho^{1/2}\|^2_{L^2}.
\end{split}
\end{equation*}
This concludes the proof of \eqref{est:Q}.
\end{proof}

With \eqref{est:Q} in hand, we can now estimate
\begin{equation*}
\begin{split}
&\sum_{k=0}^\infty\sum_{M\in2^{\Z}}\sum_{R\in2^{\Z}}M^{1+2\beta}R^{1-2\beta}\int_{R}^{2R}
\int_{0}^\infty\big|J_{\nu_k}(r\rho)b_{k}(M\rho)\chi(\rho)
\big|^2 d\rho\, dr\\\lesssim&
\sum_{k=0}^\infty\sum_{M\in2^{\Z}}\sum_{R\in2^{\Z}}M^{1+2\beta}R^{1-2\beta}Q_{k}(R,M)\\\lesssim&
\sum_{k=0}^\infty\sum_{M\in2^{\Z}}\Big(\sum_{R\in2^{\Z}, R\lesssim 1}M^{1+2\beta}R^{1-2\beta}R^{2\nu_k+1}M^{-2}
+\sum_{R\in2^{\Z}, R\gg 1}M^{1+2\beta}R^{1-2\beta} M^{-2}\Big)\big\|b_{k}(\rho)\chi\big(\tfrac{\rho}M\big)\rho^{\frac12}\big\|^2_{L^2}
\\\lesssim&
\sum_{k=0}^\infty\sum_{M\in2^{\Z}}\Big(\sum_{R\in2^{\Z}, R\lesssim 1}M^{2\beta-1}R^{2(1+\nu_k-\beta)}
+\sum_{R\in2^{\Z}, R\gg 1}M^{2\beta-1}R^{1-2\beta}\Big)\big\|b_{k}(\rho)\chi\big(\tfrac{\rho}M\big)\rho^{\frac12}\big\|^2_{L^2}.
\end{split}
\end{equation*}
Under the assumption: $\frac12<\beta<1+\nu_0 $, we sum in $R$ to get
\begin{equation}\label{equ:sumrm}
  \sum_{k=0}^\infty\sum_{M\in2^{\Z}}M^{2\beta-1}\big\|b_{k}(\rho)\chi\big(\tfrac{\rho}M\big)\rho^{\frac12}\big
  \|^2_{L^2}=\|f\|^2_{\dot H^{\beta-\frac12}_{{\A},a}(\R^2)}.
\end{equation}
  Indeed,
it follows from \eqref{funct} that
  \begin{align*}
    \LL_{{\A},a}^\frac{s}{2}f(r,\theta)=&\sum_{k=0}^\infty \psi_{k}(\theta) \int_0^\infty \rho^s J_{\nu_k}(r\rho)b_{k}(\rho) \,\rho d\rho
    =\sum_{k=0}^\infty \psi_{k}(\theta) \mathcal{H}_{\nu(k)}\big(\rho^sb_k(\rho)\big)(r).
  \end{align*}
 And so we obtain
  \begin{align*}
   \|f\|^2_{\dot H^s_{{\A},a}(\R^2)}= \big\|\LL_{{\A},a}^\frac{s}{2}f\big\|_{L^2(\R^2)}^2=  & \int_0^\infty \int_0^{2\pi}\Big|\sum_{k=0}^\infty \psi_{k}(\theta) \mathcal{H}_{\nu(k)}\big(\rho^sb_k(\rho)\big)(r)\Big|^2\;d\theta\;r\;dr\\
    =&\sum_{k=0}^\infty  \int_0^\infty \Big|\mathcal{H}_{\nu(k)}\big(\rho^sb_k(\rho)\big)(r)\Big|^2r\;dr\\
    =&\sum_{k=0}^\infty  \int_0^\infty \int_0^{2\pi}\Big|\mathcal{H}_{\nu(k)}\big(\rho^sb_k(\rho)\psi_k(\omega)\big)(r)\Big|^2\;d\theta\;r\;dr\\
    =&\sum_{k=0}^\infty \int_{\R^2}\Big|\mathcal{H}_{\nu(k)}\big(\rho^sb_k(\rho)\psi_k(\omega)\big)(r)\Big|^2\;dx.
\end{align*}
Using Lemma \ref{Lem:Hankel} $(iii)$, we get
\begin{align*}
  \big\|\LL_{{\A},a}^\frac{s}{2}f\big\|_{L^2(\R^2)}^2=  &\sum_{k=0}^\infty \int_{\R^2}\Big|\rho^sb_k(\rho)\psi_k(\omega)\Big|^2\;d\xi
    =\sum_{k=0}^\infty \int_0^\infty \Big|\rho^sb_k(\rho)\Big|^2\rho\;d\rho.
\end{align*}
Applying the unit decomposition, one has
\begin{align*}
  \big\|\LL_{{\A},a}^\frac{s}{2}f\big\|_{L^2(\R^2)}^2=  &\sum_{k=0}^\infty \int_0^\infty \Big|\sum_{M\in2^{\Z}}\chi\big(\tfrac{\rho}{M}\big)\rho^sb_k(\rho)\Big|^2\rho\;d\rho\\
    \simeq&\sum_{k=0}^\infty \sum_{M\in2^{\Z}} \int_0^\infty \Big|\chi\big(\tfrac{\rho}{M}\big)\rho^sb_k(\rho)\Big|^2\rho\;d\rho\\
    \simeq&\sum_{k=0}^\infty \sum_{M\in2^{\Z}} M^s \big\|\chi\big(\tfrac{\rho}{M}\big)b_k(\rho)\rho^\frac12\big\|_{L^2_\rho}^2.
\end{align*}
This implies \eqref{equ:sumrm}, hence we proved \eqref{local-s}, and the proof of \eqref{local-s} is complete.
\end{proof}

\subsection{Conclusion of the proof of Theorem \ref{thm:Stri}}
Let $u$ be the solution of \eqref{wave}, given by \eqref{eq:solution}.
The case $q=+\infty$ in Theorem \ref{thm:Stri} immediately follows by Spectral Theory and the Sobolev embedding in Lemma \ref{cor:sobl2}. Indeed, one has
\begin{equation*}
\begin{split}
\|u(t,z)\|_{L^\infty(\R;L^{r}(\R^2))}&\lesssim \|\LL^{\frac s2}_{{\A},a} u(t,x)\|_{L^\infty(\R;L^{2}(\R^2))}\\&\lesssim \|f\|_{\dot H^{s}_{{\A},a}(\R^2)}+\|g\|_{\dot H^{s-1}_{{\A},a}(\R^2)}\end{split}
\end{equation*}
where $s=1-\tfrac2r$ and $2\leq r<+\infty$.

Now, let $v$ be the purely magnetic wave
$$
v(t,\cdot):=\cos(t\sqrt{\mathcal L_{{\A},0}})f(\cdot)+\frac{\sin(t\sqrt{\mathcal L_{{\A},0}})}{(\sqrt{\mathcal L_{{\A},0}})}g(\cdot).
$$
By the Duhamel formula, we can hence write
\begin{equation}\label{equ:duhamel}
u(t,\cdot)=v(t,\cdot)-\int_0^t\frac{\sin{(t-\tau)\sqrt{\LL_{{\A},0}}}}
{\sqrt{\LL_{{\A},0}}}\big(\tfrac{a(\hat{x})}{|x|^2}u(\tau,\cdot)\big)\,d\tau.
\end{equation}
By \eqref{claim}, it follows that
$$
\|v(t,x)\|_{L^q(\R;L^{r}(\R^2))}\leq C\left(\|f\|_{\dot H^s_{{\A},0}}+\|g\|_{\dot H^{s-1}_{{\A},0}}\right),
$$
for $s\in\R$, any wave-admissible pair $(q,r)\in\Lambda_s^W$ as in \eqref{adm-p-w}, and for some $C>0$ independent on $f,g$.
Therefore we get
\begin{align}\label{eq:new1}
&
\|u(t,x)\|_{L^q(\R;L^{r}(\R^2))}
\\
\leq& C\left(\|f\|_{\dot H^s_{{\A},0}}+\|g\|_{\dot H^{s-1}_{{\A},0}}\right)
+\Big\|\int_0^t\frac{\sin{(t-\tau)\sqrt{\LL_{{\A},0}}}}
{\sqrt{\LL_{{\A},0}}}\big(\tfrac{a(\hat{x})}{|x|^2}u(\tau,x)\big)d\tau\Big\|_{_{L^q(\R;L^{r}(\R^2))}}
\nonumber
\end{align}
Now our main task is to estimate the $TT^*$-term
\begin{equation}\label{est:inh}
\begin{split}
\Big\|\int_0^t\frac{\sin{(t-\tau)\sqrt{\LL_{{\A},0}}}}
{\sqrt{\LL_{{\A},0}}}\big(\tfrac{a(\hat{x})}{|x|^2}u(\tau,x)\big)d\tau\Big\|_{_{L^q(\R;L^{r}(\R^2))}}.
\end{split}
\end{equation}
Notice that if the set $\Lambda_s^W$ is not empty, we must have $0\leq s<1$. And when $s=0$, we must have $(q,r)=(+\infty,2)$. Hence we only need to study the range $0<s<1$. We will treat separately the following two cases:
\begin{enumerate}
  \item[(i)] $0<s<\min\big\{1,\tfrac12+\nu_0\big\}$,
  \item[(ii)]  $\tfrac12+\nu_0\leq s<1$ with $\nu_0<\tfrac12$.
\end{enumerate}

{\bf Case 1: $0<s<\min\big\{1,\tfrac12+\nu_0\big\}.$}
Define the operator
$$T: L^2(\R^2)\to L^2(\R;L^2(\R^2)), \quad T f= r^{-\beta}e^{it\sqrt{\LL_{{\A},0}}}\LL_{{\A},0}^{\frac12(\frac12-\beta)} f.$$
Thus from the proof of the local smoothing estimate, it follows that $T$ is a bounded operator.
By duality, we obtain that for its adjoint $T^*$ $$T^*: L^2(\R;L^2(\R^2))\to L^2, \quad T^* F=\int_{\tau\in\R}\LL_{{\A},0}^{\frac12(\frac12-\beta)} e^{-i\tau\sqrt{\LL_{{\A},0}}}  r^{-\beta}  F(\tau)d\tau$$
is also bounded. Define the operator
$$B: L^2(\R;L^2(X))\to L^q(\R;L^r(\R^2)), \quad B F=\int_{\tau\in\R} \frac{e^{i(t-\tau)\sqrt{\LL_{{\A},0}}}}{\sqrt{\LL_{{\A},0}}} r^{-\beta}F(\tau)d\tau.$$
Hence by the Strichartz estimate  \eqref{claim} with $s=\frac32-\beta$, one has
\begin{equation}\label{BF}
\begin{split}
&\|B F\|_{L^q(\R;L^r(\R^2))}\\=&\Big\| e^{i t\sqrt{\LL_{{\A},0}}}\int_{\tau\in\R}\frac{e^{-i\tau\sqrt{\LL_{{\A},0}}}}{\sqrt{\LL_{{\A},0}}} r^{-\beta} F(\tau)d\tau\Big\|_{L^q(\R;L^r(\R^2))}\\
\lesssim& \Big\|\int_{\tau\in\R}\frac{e^{-i\tau\sqrt{\LL_{{\A},0}}}}{\sqrt{\LL_{{\A},0}}} r^{-\beta} F(\tau)d\tau\Big\|_{\dot {H}_{{\A},0}^{\frac32-\beta}(\R^2)}=\|T^*F\|_{L^2}\lesssim \|F\|_{L^2(\R;L^2(\R^2))}.
\end{split}
\end{equation}
Now we estimate \eqref{est:inh}. Note that
$$\sin(t-\tau)\sqrt{\LL_{{\A},0}}=\frac{1}{2i}\big(e^{i(t-\tau)\sqrt{\LL_{{\A},0}}}-e^{-i(t-\tau)\sqrt{\LL_{{\A},0}}}\big),$$
thus by \eqref{BF}, we have a minor modification of \eqref{est:inh}
\begin{equation*}
\begin{split}
&\Big\|\int_\R\frac{\sin{(t-\tau)\sqrt{\LL_{{\A},0}}}}
{\sqrt{\LL_{{\A},0}}}\big(\tfrac{a(\hat{x})}{|x|^2}u(\tau,x)\big)d\tau\Big\|_{L^q(\R;L^{r}(\R^2))}\\\lesssim & \big\|B\big(r^{\beta}\tfrac{a(\hat{x})}{|x|^2}u(\tau,x)\big)\big\|_{L^q(\R;L^r(\R^2))}\\
\lesssim& \|r^{\beta-2}u(\tau,x)\|_{L^2(\R;L^2(\R^2))}
\\\lesssim& \|f\|_{\dot H^{\frac32-\beta}_{{\A},a}(\R^2)}+\|g\|_{\dot H^{\frac12-\beta}_{{\A},a}(\R^2)}
\end{split}
\end{equation*}
where we use the local smoothing estimate in Proposition \ref{prop:loc} again in the last inequality and we need $1-\nu_0<\beta<3/2$ such that $1/2<2-\beta<1+\nu_0$.
Therefore the above statement holds for all $\max\{1/2,1-\nu_0\}<\beta<3/2$.
By the Christ-Kiselev lemma \cite{CK}, thus we have showed that for $q>4$ and $(q,r)\in\Lambda^W_s$ with $s=\frac32-\beta$
\begin{equation}
\begin{split}
\eqref{est:inh}\lesssim \|f\|_{\dot H^{s}_{{\A},a}(\R^2)}+\|g\|_{\dot H^{s-1}_{{\A},a}(\R^2)}.
\end{split}
\end{equation}
Therefore we have proved all  $(q,r)\in\Lambda^W_s$ when $s$ satisfies $0<s<\min\big\{1,\tfrac12+\nu_0\big\}$.
\vskip 0.2in

{\bf Case 2}: $\tfrac12+\nu_0\leq s<1$ with $\nu_0<\tfrac12$. To this end, we split the initial data into two parts: one is projected to $\mathcal{H}^k$ with $k\leq 1+\nu_0$ and the other is the remaining terms. Without loss of generalities, we assume $g=0$ and
  divide $f=f_{l}+f_{h}$ where $f_{h}=f-f_{l}$ and
 \begin{equation}
 f_{l}(x)=\sum_{k=0}^{1}a_{k}(r)\psi_k(\theta).
 \end{equation}
 For the part involving $f_{h}$, we can repeat the argument of {\bf Case 1}. In this case, as remarked in Remark \ref{rem:loc-s}, we can use Proposition \ref{prop:loc} with $1/2<2-\beta<2+\nu_0$.
 Thus we obtain the Strichartz estimate on $e^{it\sqrt{\LL_{{\A},a}}}f_{h}$ for $\Lambda_{s}^W$ with $s\in[\frac12+\nu_0,1)$.

 Next we consider the Strichartz estimate on $e^{it\sqrt{\LL_{{\A},a}}}f_{l}$. We follow the argument of \cite{PST} which treated a radial function.
Recall from \eqref{funct}
 \begin{equation}
\begin{split} e^{it\sqrt{\LL_{{\A},a}}}u_{0,l}(x)&=
\sum_{k=0}^1\psi_k(\theta)\int_0^\infty J_{\nu(k)}(r\rho)e^{it\rho}\mathcal{H}_{\nu(k)}(a_k)\rho \;d\rho,\\
&=\sum_{k=0}^1\psi_k(\theta)\mathcal{H}_{\nu(k)}[e^{it\rho}\mathcal{H}_{\nu(k)}(a_k)](r).
\end{split}
\end{equation}
Since $\psi_k(\theta)\in L^r(\mathbb{S}^1),$ we get
 \begin{equation}
\begin{split} \|e^{it\sqrt{\LL_{{\A},a}}}u_{0,l}\|_{L^q(\R;L^r(\R^2))}&\leq C
\sum_{k=0}^1\left\|\mathcal{H}_{\nu(k)}[e^{it\rho}\mathcal{H}_{\nu(k)}(a_k)](r)\right\|_{L^q(\R;L^r_{rdr})}.
\end{split}
\end{equation}
Recall $\mathcal{H}_0\mathcal{H}_0=Id$, then it suffices to estimate
 \begin{equation}
\begin{split}
\sum_{k=0}^1\left\|(\mathcal{H}_{\nu(k)}\mathcal{H}_0)\mathcal{H}_0[e^{it\rho}\mathcal{H}_0(\mathcal{H}_0\mathcal{H}_{\nu(k)})
(a_k)](r)\right\|_{L^q(\R;L^r_{rdr})}.
\end{split}
\end{equation}
For our purpose, we recall \cite[Theorem 3.1]{PST} which claimed that the operator $\mathcal{K}^0_{\mu,\nu}:=\mathcal{H}_\mu\mathcal{H}_\nu$ is continuous on $L^p_{r^{n-1}dr}([0,\infty))$ if
$$\max\big\{\tfrac1n\big(\tfrac{n-2}2-\mu\big),0\big\}<\tfrac1p<\min\big\{\tfrac1n\big(\tfrac{n-2}2+\nu+2\big),1\big\}.$$
Notice $n=2$, we obtain that $\mathcal{K}^0_{0,\nu}$ and $\mathcal{K}^0_{\nu,0}$ are bounded in
$L^p_{rdr}([0,\infty))$ provided $p>2$ and $\nu>0$. On the other hand,
$\mathcal{H}_0[e^{it\rho}\mathcal{H}_0]$ is a classical half-wave propagator in the radial case which has Strichartz estimate with $(q,r)\in\Lambda_s^W$. In sum, for $(q,r)\in\Lambda_{s}^W$, we have
 \begin{equation}
\begin{split}
&\|e^{it\sqrt{\LL_{{\A},a}}}f_{l}\|_{L^q(\R;L^r(\R^2))}\\\leq& C\sum_{k=0}^1\left\|(\mathcal{H}_{\nu(k)}\mathcal{H}_0)\mathcal{H}_0[e^{it\rho}\mathcal{H}_0(\mathcal{H}_0\mathcal{H}_{\nu(k)})
(a_k)](r)\right\|_{L^q(\R;L^r_{rdr})}\\
\leq& C\sum_{k=0}^1\left\|(\mathcal{H}_0\mathcal{H}_{\nu(k)})(a_k)](r)\right\|_{\dot H^s_{{\A},a}}\leq C\left(\sum_{k=0}^1\left\|a_{k}(r)\right\|^2_{\dot H^s_{{\A},a}}\right)^{1/2}\leq C\|f_{l}\|_{\dot H^s_{{\A},a}}.
\end{split}
\end{equation}
In the second inequality, we use \cite[Theorem 3.8]{PST}.
Therefore, we conclude the proof of Theorem \ref{thm:Stri}.

\begin{center}

\end{center}

\end{document}